\documentclass[12pt]{amsart}

\usepackage{amsmath,amsfonts,amsthm,amssymb ,textcomp}
\usepackage[margin=2cm]{geometry}

\usepackage{hyperref}

\def \udot {{}^{\textstyle .}}
\usepackage{color}

\newtheorem{Thm}{Theorem}[section]
\newtheorem{Cor}[Thm]{Corollary}

\newtheorem{Prop}[Thm]{Proposition}

\newtheorem{Lem}[Thm]{Lemma}

\theoremstyle{definition}
\newtheorem{Def}[Thm]{Definition}
\newtheorem{Not}[Thm]{Notation}

\theoremstyle{remark}

\numberwithin{equation}{section}

\newcommand{\Aut}{\operatorname{Aut}}

\newcommand{\Hom}{\operatorname{Hom}}
\newcommand{\Mor}{\operatorname{Mor}}

\newcommand{\Syl}{\operatorname{Syl}}

\newcommand{\Out}{\operatorname{Out}}
\newcommand{\Dih}{\operatorname{Dih}}

\newcommand{\Inn}{\operatorname{Inn}}

\newcommand{\Stab}{\operatorname{Stab}}

\newcommand{\Ly}{\operatorname{Ly}}

\def \a {\alpha}\def \b {\beta}

\newcommand{\Sym}{\operatorname{Sym}}
\newcommand{\GSp}{\operatorname{GSp}}

\newcommand{\Alt}{\operatorname{Alt}}
\newcommand{\M}{\operatorname{M}}
\newcommand{\HN}{\operatorname{HN}}
\newcommand{\Mo}{\operatorname{M}}
\newcommand{\BM}{\operatorname{B}}
\newcommand{\Gee}{\operatorname{G}}
\newcommand{\Sp}{\operatorname{Sp}}
\newcommand{\Mat}{\operatorname{Mat}}
\newcommand{\PSL}{\operatorname{PSL}}
\newcommand{\GL}{\operatorname{GL}}
\newcommand{\SL}{\operatorname{SL}}

\newcommand{\PSp}{\operatorname{PSp}}

\renewcommand{\epsilon}{\varepsilon}
\renewcommand{\bar}{\overline}
\renewcommand{\hat}{\widehat}
\renewcommand{\leq}{\leqslant}
\renewcommand{\geq}{\geqslant}

\newcommand{\X}{\mathcal{X} }

\newcommand{\F}{\mathcal{F}}

\newcommand{\G}{\mathcal{G}}

\newcommand{\U}{\mathcal{U}}
\renewcommand{\P}{\mathcal{P}}

\newcommand{\E}{\mathcal{E}}
\newcommand{\W}{\mathcal{W}}

\newcommand{\C}{\mathcal{C}}

\def \Frob {\mathrm{Frob}}

\def \PSU{\mathrm{PSU}}

\begin{document}

\title{Fusion systems over a Sylow $p$-subgroup of $\Gee_2(p)$}

\author{Chris Parker}
\address{
School of Mathematics\\
University of Birmingham\\
Edgbaston\\
Birmingham B15 2TT\\ United Kingdom
} \email{c.w.parker@bham.ac.uk}

\author{Jason Semeraro}
\address{Heilbronn Institute for Mathematical Research, Department of Mathematics, University of Bristol,  United Kingdom}
\email{js13525@bristol.ac.uk}

\begin{abstract} For $S$ a Sylow $p$-subgroup of the group $\Gee_2(p)$ for $p$ odd, up to isomorphism of fusion systems,
we determine all saturated fusion systems $\F$ on $S$ with $O_p(\F)=1$. For $p \ne 7$, all such fusion systems are realized by  finite groups whereas for $p=7$ there are 29 saturated fusion systems   of which 27 are exotic.
\end{abstract}



\maketitle

\section{Introduction}

Let $p \ge 3$ and $S$ be a Sylow $p$-subgroup of the group $\Gee_2(p)$. The purpose of this paper is
to give a complete classification of all saturated fusion systems $\F$ over $S$ with $O_p(\F)=1$.
This may be viewed as a contribution to a program which aims to classify all saturated fusion
systems over maximal unipotent subgroups of finite groups of Lie Type of rank 2 and is thus a
natural continuation of work carried out in \cite{C,  HS,RV}. In a different direction, when $p\ge 5$ our paper
contributes to the problem of listing all saturated fusion systems $\F$ over a Sylow $p$-subgroup
with an  extraspecial $p$-subgroup of index $p$, currently under investigation by the first
author and Raul Moragues Moncho.  An infinite family of such fusion systems was discovered recently by the
first author and Stroth \cite{PS}, and it is the $p$-group underlying the smallest member of this
family on which we focus our attention. It will also form part of the classification of fusion systems of sectional $p$-rank $4$ for odd primes $p$. All of these contributions add to our knowledge of saturated fusion systems defined on $p$-groups for odd primes $p$ and so extend our  understanding of how exotic fusion systems arise at odd primes \cite[Problem 7.4]{AKO} and \cite[Problem 7.6]{AO}.

When $p \ge 5$, the problem naturally breaks into three stages. First in Section \ref{s:constructs} we give a presentation for $S$ and provide a concrete description of its action on the unique extraspecial subgroup $Q$ of index $p$. Using this description, if $\F$ is a saturated fusion system on $S$ we whittle down the possibilities for the $\F$-essential subgroups in Section \ref{s:candidates} by using results concerning the way in which automorphisms of a $p$-group act on various subgroups and conditions on the existence of certain lifts of automorphism groups which arise because of the saturation axiom.

Armed with a small list of possibilities for the $\F$-essential subgroups, in Section \ref{s:determining} we proceed to analyse the various combinations of essential subgroups and morphisms for $\F$ which have the potential to lead to a saturated fusion system. Here we are especially reliant on a short list of possibilities for the group $\Aut_\F(Q)$ which follows from some results obtained by the second author together with Craven and Oliver in \cite{COS}. One issue that arises during this stage is the question of whether or not a fusion system is uniquely determined  by the above data. We develop some techniques to answer this, especially relying on some delicate calculations of automorphism groups carried out at the end of Section \ref{s:constructs}. Generally, our scheme is as follows: suppose for simplicity that we are in the typical case where there are just two essential subgroups $Q$ and $R$ in $\F$, which are the unipotent radical subgroups of proper parabolic subgroups  of $\Gee_2(p)$ lying in $S$. In this generic case we know that $\Out_\F(R)$ contains a normal subgroup isomorphic to $\SL_2(p)$ by Lemma~\ref{l:AFR1}.  The saturation axiom and the presence of this subgroup of $\Out_\F(R)$ combine to give the existence of certain morphisms in $\Aut_\F(S)$ and then in  $\Aut_\F(Q)$ by restriction. Now we use just the existence of these automorphisms to  determine the possibilities for the structure of $\Aut_\F(Q)$ as a subgroup of $\Aut(Q)$ containing $\Aut_S(Q)$. Using the Model Theorem \cite[Theorem I.4.9]{AKO},  we discover that $N_\F(Q)$ and $\Aut_\F(S)$ are uniquely determined. Since we are allowed to adjust a fusion system by morphisms in $\Aut(S)$ while preserving its isomorphism type, we may from this point on assume that $\Aut_\F(S)$ is a fixed subgroup of $\Aut(S)$ identified as a subgroup of $\Aut_B(S)$ where $B$ is as defined in Section \ref{s:constructs}. This allows us to make explicit calculations with elements of $\Aut_\F(S)$. Next we consider the subgroup  $N_{\Aut_\F(R)}(\Aut_S(R))$ given by restricting the morphisms in  $\Aut_\F(S)$  to $R$.  Employing Lemma~\ref{l:runique}, we already know that, in these favourable  circumstances, in $\Aut(R)$ there is a unique subgroup $X$ containing $N_{\Aut_\F(R)}(\Aut_S(R))$ with $\Aut_{S}(R) \in \Syl_p(X)$ and $O^{p'}(X) \cong \SL_2(p)$. Thus we must have $\Aut_\F(R)=X$ and this is uniquely determined as a subgroup of $\Aut(R)$. Thus we see that all the morphisms of the essential subgroups of $\F$  are given uniquely by the group $\Aut_\F(S)$ and so the fusion systems are uniquely determined.

In the final stage, in Section \ref{s:final}, we examine each candidate fusion system $\F$ in turn and establish (a) its existence, (b) whether it is saturated and (c) whether it is realizable as the fusion system of a finite (almost simple) group. Here the fact that the fusion systems are uniquely determined by the structure of the automorphism groups of their essential subgroups is used implicitly. In all but finitely many cases, we obtain an affirmative answer to (a) and (b) from an affirmative answer to (c). In the remaining cases, it is always possible to realize $\F$ as the fusion system of a free amalgamated product of finite groups and saturation is established using the geometry of the associated coset graph (Theorem \ref{They are saturated}).

When $p=3$, the two unipotent radical subgroups of $\Gee_2(3)$ are isomorphic so that although the overall strategy of the proof is the same, the individual arguments are  somewhat different. In addition, in this case there is only one group to consider and we can support our arguments by computer calculations \cite{magma} especially in the proof of uniqueness of the fusion systems.  This case is treated in the final section.

Our main theorem is as follows:

\begin{Thm}\label{t:main1}
Suppose that $p \geq 3$, $S$ is a Sylow $p$-subgroup of $\Gee_2(p)$ and $\F$ is a saturated fusion system over $S$ with $O_p(\F)=1$. Then either $\F$ is isomorphic to the fusion system of $\Gee_2(p)$, $\Aut(\Gee_2(3))$ or $p \in \{5,7\}$ and $\F$ is isomorphic to one of 32 examples tabulated in Table~\ref{table1}. Furthermore, each of the  fusion systems given in Table~\ref{table1}  is saturated.
\end{Thm}

The examples described in Table~\ref{table1} include  the fusion systems of the sporadic simple groups $\Ly$, $\HN$, $\mathrm B$, the almost simple group $\Aut(\HN)$ (all for $p=5$) and the sporadic simple group $\Mo$ when $p=7$. It also includes  27  exotic  fusion systems which all occur when $p=7$.  Two of the exotic systems  were discovered by Parker and Stroth \cite{PS} and the remainder are new to this article. They all are in some way related to the Monster sporadic simple group, though it is \emph{not} the case that the Monster is ``universal" in the sense that it ``contains" all the smaller examples.  This is somehow a subtle point. The fact is, and this plays no part in the classification, that in $\GSp_4(7)$, the subgroup $3 \times 2\udot \Sym(7)$  does not contain $\GL_2(7)$ but rather only a half of this group and so the fusion system that comes from $\Gee_2(7)$ is not contained in the fusion system determined by the Monster when $p=7$.

\begin{Cor}
Suppose that $p \geq 3$, $S$ is a Sylow $p$-subgroup of $\Gee_2(p)$ and $\F$ is a saturated fusion system over $S$ with $O_p(\F)=1$. Then either $\F$ is realized by a finite group or $p=7$ and  $\F$ is one of 27 fusion systems listed in Table \ref{table1}.
\end{Cor}

We close the introduction with a few words about our notation. We use \cite{AKO,Cr1, GOR} for standard group theoretic and fusion theoretic  conventions. Particularly we use \cite{AKO,Cr1} as a sources for the  introduction of fusion systems in Section~\ref{s:pgroup}. The field of order $p$ is denoted by $\mathbb F_p$, the symmetric and alternating groups of degree $n$ are denoted by $\Alt(n)$ and $\Sym(n)$ respectively and other than that   we follow classical nomenclature for the finite simple groups and their near relatives. The Frobenius group of order $n$ is written as $\Frob(n)$ and cyclic groups are mostly represented just by their order. The notation $2^{1+4}_-$ denotes an extraspecial group of $-$-type and order $2^5$ and, for $p$ odd,  $p^{1+2}_+$  is extraspecial of order $p^3$ and  exponent $p$. We use $G=A\circ B$ to indicate that $G$ is a central product of the groups $A$ and $B$. We follow the  \textsc{atlas} conventions for group extensions. This means that an ``upper" dot informs the reader that an extension is non-split. When we write $G \sim A.B$ we read that $G$ has a normal subgroup isomorphic to $A$ and a corresponding  quotient isomorphic to $B$. This provides a handy but inaccurate description of group structures.  In our case, each time we use this notation the groups will be determined uniquely up to isomorphism as a subgroup of $\GSp_4(p)$ or $\GL_2(p)$. We   point out that the notation $\SL_2(7).2$ will denote the unique normal subgroup of $\GL_2(7)$ of index $3$.

\section{Preliminaries: fusion systems and group theory}\label{s:pgroup}

We begin by recalling the definition of a fusion system. For a group $G$, $p$-subgroup $S$ of $G$ and $P,Q \le S$ define $$N_G(P,Q)=\{g \in G \mid P^g \le Q\} \mbox{ and } \Hom_G(P,Q)=\{c_g \mid g \in N_G(P,Q)\},$$ where $c_g$ is the conjugation map induced by $g$: $$c_g: x \mapsto g^{-1}xg.$$ Define  $\F_S(G)$ to be the category with objects all the subgroups of $S$, and for objects $P$ and $Q$ of $\F_S(G)$,  the set of morphisms from $P$ to $Q$ is   $$\Mor_{\F_S(G)}(P,Q)=\Hom_G(P,Q).$$ Then $\F_S(G)$ is an example of a \textit{fusion system} on $S$ as defined, for example, in \cite[Definition 2.1]{AKO}.  If $S$ is a finite $p$-group and $\F$ is  a fusion system on $S$ we say that $\F$ is \textit{realizable} if there exists a \textit{finite} group $G$ with $S \in \Syl_p(G)$ such that $\F=\F_S(G)$. Otherwise $\F$ is said to be \textit{exotic.} If $P \le S$, then define the set of \emph{$\F$-conjugates} of $P$ to be $$P^\F=\{P\alpha \mid \alpha \in \Hom_\F(P,S)\}$$  and similarly, for $g \in S$,  we use  $$g^\F=\{g\alpha \mid \alpha \in \Hom_\F(\langle g \rangle, S)\}$$ for the set of images of $g$ under morphisms in $\F$.    For $P \le S$, we put $\Aut_\F(P) = \Mor_\F(P,P)$,  $\Aut_S(P) = \Hom_S(P,P)$, $\Inn(P)$ the inner automorphisms of $P$ and $\Out_\F(P)= \Aut_\F(P)/\Inn(P)$. Similarly $\Out_S(P)= \Aut_S(P)/\Inn(P)$. Note that $\Aut_\F(Q) \cong \Aut_\F(P)$ for each $Q \in P^\F$. The set  of all morphisms in $\F$ is denoted by $\Mor(\F)$. Two fusion systems $\F$ and $\F'$ on $S$ are isomorphic if there exists $\alpha \in \Aut(S)$ such that for all $P,Q \le S$, $$\Hom_{\F'}(P\alpha,Q\alpha)=\{\alpha^{-1}|_{P\alpha} \theta \alpha\mid\theta \in \Hom_\F(P,Q)\}.$$ We write $\F \cong \F'$ or $\F' = \F^\alpha$ if we wish to specify $\alpha$.   A proper subgroup $H < G$ of a finite group $G$ is \textit{strongly $p$-embedded} in $G$ if $p$ divides $ |H|$ and $p$  does not divide $|H \cap H^g|$ for each $g \in G \setminus H.$ The next definition summarizes the main concepts we will need when dealing with fusion systems:

\begin{Def}
Let $\F$ be a fusion system on a finite $p$-group $S$ and $P,Q \le S$. Then,
\begin{itemize}
\item[(a)] $P$ is \textit{fully $\F$-normalized} provided $|N_S(P)| \ge |N_S(Q)|$ for all $Q \in \P^\F$;
\item[(b)] $P$ is \textit{fully $\F$-centralized} provided $|C_S(P)| \ge |C_S(Q)|$  for all $Q \in \P^\F$;
\item[(c)]   $P$ is \textit{fully $\F$-automized} provided $\Aut_S(P) \in \Syl_p(\Aut_\F(P))$;
\item[(d)] $P$ is \textit{$\F$-centric} if $C_S(Q)=Z(Q)$ for all $Q \in P^\F$;
\item[(e)] $P$ is \textit{$\F$-essential} if $P < S$, $P$ is $\F$-centric and  fully $\F$-normalized and $\Out_\F(P)$ contains a strongly $p$-embedded subgroup; write $\E_\F$ (or simply $\E$) to denote the set of $\F$-essential subgroups of $\F$;
\item[(f)] $P$ is \textit{strongly $\F$-closed} if for each $g \in P$, $g^\F \subseteq P$;
\item[(g)] if $\alpha \in \Hom_\F(P,Q)$ is an isomorphism, $$N_\alpha=\{g \in N_S(P) \mid \alpha^{-1}c_g \alpha \in \Aut_S(Q)\}$$ is the $\alpha$-\emph{extension control subgroup} of $S$;
\item[(h)] $Q$ is $\F$-\emph{receptive} provided for all isomorphisms $\alpha \in \Hom_\F(P,Q)$, there exists $\widetilde{\alpha} \in \Hom_\F(N_\alpha,S)$ such that $\widetilde{\alpha}|_P =\alpha;$
\item[(i)] $P$ is $\F$-\emph{saturated} provided there exists $Q \in P^\F$ such that $Q$ is simultaneously
\begin{itemize}
\item[(1)] fully $\F$-automized; and
\item[(2)] $\F$-receptive;
\end{itemize}
\item[(j)] $\F$ is \emph{saturated} if every subgroup of $S$ is $\F$-saturated.
\end{itemize}
\end{Def}

Saturated fusion systems are the main focus of study. Suppose that $\F$ is saturated. Then,  by \cite[Lemma 2.6 (c)]{AKO}, a subgroup $Q$ of $S$ is fully $\F$-normalized if and only if it is fully $\F$-automized and $\F$-receptive.  In particular,  $\F$-essential subgroups are both  fully $\F$-automized and $\F$-receptive.
We shall exploit the saturation property as follows.  Suppose that $Q$ is $\F$-receptive.   If $\alpha \in N_{\Aut_\F(Q)}(\Aut_S(Q))$, then $$N_\a= \{x\in N_S(Q)\mid \a^{-1}c_x\alpha \in N_{\Aut_\F(Q)}(\Aut_S(Q))\} =N_S(Q)$$ and so there exists $\widetilde \a \in \Hom_\F(N_S(Q),S)$ extending $\alpha$. Since $$N_S(Q)\widetilde \alpha \le N_S(Q\widetilde \alpha )= N_S(Q\alpha)= N_S(Q),$$ we have $\widetilde \a \in \Aut_\F(N_S(Q))$.  Therefore every $\alpha \in N_{\Aut_\F(Q)}(\Aut_S(Q))$ extends to an element of $\Aut_\F(N_S(Q))$. We shall often use the fact that $O_p(\Aut_\F(E))= \Inn(E)$ if $E$ is $\F$-essential which follows as $\Out_\F(E)$ has a strongly $p$-embedded subgroup.

Recall that when $G$ is a finite group and $S \in \Syl_p(G)$, we have that $\F_S(G)$ is saturated. If $X$ is a set of injective morphisms between various subgroups of $S$, then we may define $\langle X \rangle$ to be the fusion system obtained by intersecting all the fusion systems on $S$ which have the members of $X$ as morphisms.

The next result  is commonly referred to in the literature as ``Alperin's Theorem."

\begin{Thm}\label{t:alp}
Let $\F$ be a saturated fusion system on a finite $p$-group $S$. Then
 $$\F=\langle \Aut_\F(E) \mid E \in \E_\F \cup \{S\} \rangle.$$
\end{Thm}

For $Q$ a subgroup of $S$, we take the definition of $N_\F(Q)$ from \cite[Definition I.5.3]{AKO} and note that when $Q$ is fully $\F$-normalised, $N_\F(Q)$ is a saturated fusion system on $N_S(Q)$ by \cite[Theorem I.5.5]{AKO}.

 A subgroup $Q \le S$ is \textit{normal} in $\F$ if  and only if $N_\F(Q)=\F$ which is if and only if $Q \le \bigcap_{P\in \E_\F}P$  and, for $P \in \E_\F \cup \{S\}$, $Q$ is $\Aut_\F(P)$-invariant (see \cite[Proposition 4.5]{AKO}.) The subgroup $O_p(\F)$ of $S$ is the largest normal subgroup of $\F$. Recall the definition of $O^{p'}(\F)$ which can be found  in \cite[Section 7.5]{Cr1}, and that a subsystem of $\F$ has \textit{index prime to $p$}  (or \textit{$p'$-index}) in $\F$ if and only if it contains $O^{p'}(\F)$.  Define $$O^{p'}_*(\F)= \langle O^{p'}(\Aut_\F(R)) \mid R \le S\rangle.$$
Then put \begin{eqnarray*}\Aut_{\F}^0(S) &=&\langle \alpha \in \Aut_\F(S)\mid \alpha|_P \in \Hom_{O^{p'}_*(\F)}(P,S) \text { for some }\F\text{-centric  } P\le S\rangle,\end{eqnarray*}
and set $$\Gamma_{p'}(\F) = \Aut_\F(S)/\Aut_{\F}^0(S).$$ We have the following:

\begin{Thm}\label{Op'} Suppose that $\F$ is a saturated fusion system on $S$.  Then there is a one-to-one correspondence between saturated sub-fusion systems of $\F$ on $S$ of index prime to $p$ and subgroups of $\Gamma_{p'}(\F)$.
\end{Thm}

\begin{proof}
See \cite[Theorem 7.7]{AKO}.
\end{proof}

When proving that a fusion system is saturated, the following theorem is a basic tool:

\begin{Thm}\label{t:alpconv}
Let $\F$ be a fusion system on a finite $p$-group $S$ and let $\C$ denote the set of all $\F$-centric subgroups. Suppose that $\F=\langle \Aut_\F(P) \mid P \in \C \rangle.$ If $P$ is $\F$-saturated for each $P \in \C$, then  $\F$ is saturated.
\end{Thm}

\begin{proof}
See \cite[Theorem A]{BCGLO}.
\end{proof}

Sometimes we consider the fusion system determined by $G$, the universal completion of an amalgam $G_1 \ge G_{12} \le G_2$ of finite groups with $S$ a Sylow  $p$-subgroup of either $G_1$ or $G_2$ (or both). We define the \emph{coset graph} of $G_1$ and $G_2$ in $G$ to be the graph
$\Gamma=\Gamma(G,G_1,G_2,G_{12})$ which has $$V(\Gamma)=\{G_ig \mid g \in G, i \in \{1,2\}\} \mbox{ and } E(\Gamma)=\{\{G_1g,G_2h\} \mid G_1g \cap G_2h \neq \emptyset, g,h \in G\}.$$ Since $G$ is the universal completion of the amalgam, $\Gamma$ is a tree \cite[Theorem 6]{Serre}. It is easy to verify that $G$ acts on $\Gamma$ by right multiplication.  We shall always consider amalgams which are ``simple" in the sense that no normal subgroup of $G$ is contained in $G_{12}$. In this case, the action of $G$ on $\Gamma$ is faithful.  Finally, we note that the stabiliser of the vertex $G_ig$ is just $G_i^g$ and that the edge-stabilizers are $G$-conjugate to $G_{12}$.

The following result shows that the saturation of $\F_S(G)$ is determined to some extent by the graph
$\Gamma$ and the action of $G$ on it. The proof of this result, which is taken from \cite{P},  requires that we remember that when a finite group acts on a tree without exchanging the vertices of some edge, then it fixes a vertex.

\begin{Thm}\label{t:satamalg}
Let $\mathcal A= (G_1 \ge G_{12} \le G_2)$ be an amalgam of finite groups, assume that $\Syl_p(G_{12}) \subseteq \Syl_p(G_2)$ and fix $S_i \in \Syl_p(G_i)$ with $S_2 \le S_1$.  Assume that $G=G_1 *_{G_{12}} G_2$ is the   universal completion of $\mathcal A$ and write $\Gamma=\Gamma(G,G_1,G_2,G_{12})$ for    the coset graph. Suppose that:
\begin{enumerate}
\item[(a)] for all $\F_{S_1}(G)$-centric subgroups $P$ of $S_1$, $\Gamma^P$ is finite; and
\item[(b)] each $\F_{S_i}(G_i)$-essential subgroup is $\F_{S_1}(G)$-centric.
\end{enumerate}
Then $\F_{S_1}(G)$ is saturated.
\end{Thm}

\begin{proof}
Since, for $i=1,2$,  $G_i$ is finite, $\F_{S_i}(G_i)$ is a saturated fusion system on $S_i$ and hence $\F_{S_i}(G_i)$ is generated by the  $\F_{S_i}(G_i)$-automorphisms of $S$ and $\F_{S_i}(G_i)$-automorphisms of the $\F_{S_i}(G_i)$-essential subgroups by Alperin's Theorem. Since, by \cite[Theorem 1]{Robinson}, $$\F_{S_1}(G)=\langle \F_{S_1}(G_1), \F_{S_2}(G_2) \rangle,$$ (b) implies that $\F_{S_1}(G)$ is generated by the collection of  $\Aut_{\F_i}(X)$ for $X$ an $\F_{S_1}(G)$-centric subgroup of $S_1$. Thus, by Theorem \ref{t:alpconv}, $\F_{S_1}(G)$ is saturated provided each $\F_{S_1}(G)$-centric subgroup is $\F_{S_1}(G)$-saturated.

Put $\F= \F_{S_1}(G)$ and assume that $P\le S_1$ is an $\F$-centric subgroup of $S_1$. Since $G_1,G_2$ and $\Gamma^P$ are finite, the subgroup  $K$ of $N_G(P)$ which fixes every vertex of $\Gamma^P$ is finite.
 Now $N_G(P)/K$ embeds into $\Aut(\Gamma^P)$ and so is also finite.  Thus $N_G(P)$ is finite and so $N_G(P)$ is contained in $\mathrm{Stab}_G(\a)$  for some $\a \in \Gamma^P$. Therefore $N_G(P)$ is $G$-conjugate to a subgroup of either $G_1$ or $G_2$. Hence we may choose a $G$-conjugate $P^f$ of $P$ so that either
 $$N_G(P^f) \le G_1\text{ and }R\in  \Syl_p(N_G(P^f)) \text{ has } R\le S_1, $$ or
   $$N_G(P^f) \le G_2\text{ and }R\in  \Syl_p(N_G(P^f)) \text{ has } R\le S_2. $$
 Thus $$\Aut_{S_1}(P^f)= RC_G(P^f)/C_G(P^f) \in \Syl_p(\Aut_G(P^f))$$ and hence $P^f$ is fully $\F$-automized.

 It remains to prove that every $\F$-centric subgroup $P$ in $S_1$ is $\F$-receptive. So assume that $c_g \in \Hom_\F(U,P)$ is an isomorphism and define
$$N=N_{c_g} = \{h \in N_{S_1}(U)\mid c_{g^{-1}hg} \in \Aut_{S_1}(P)\}.$$ Then $$N^gC_G(P) \le N_{S_1}(P)C_G(P).$$ Since $P$ is $\F$-centric and $C_G(P)$ is finite, $$C_G(P) = Z(P) \times O_{p'}(C_G(P)).$$ Thus $N_{S_1}(P) \in \Syl_p(N_{S_1}(P)C_G(P))$.  So there exists $x \in C_G(P)$ such that $N^{gx} \le N_{S_1}(P)$. Set $y= gx$. Then $c_{y}\in \Hom_G(N,N_{S_1}(P))$  and $c_y$ extends $c_g \in \Hom(U,P)$.  We have shown that $P$ is $\F$-receptive. In particular, $P^f$ as in the previous paragraph is both fully $\F$-automized and $\F$-receptive.  Thus $P$ is $\F$-saturated. This completes the proof.
\end{proof}

We will also need the following result from \cite{S} which gives conditions under which one can enlarge a saturated fusion system on a $p$-group $S$ to form a new saturated fusion system, by adding morphisms of certain subgroups.

\begin{Thm}\label{t:proofsat}
Let $\F_0$ be a saturated fusion system on a finite $p$-group $S$. For $1 \leq  i \leq  m$, let $W_i  \le S$ be a fully $\F_0$-normalized subgroup with $W_i\varphi \not\le W_j$ for each $\varphi \in \Hom_{\F_0}(W_i, S)$ and $i \neq j$. Set $K_i = \Out_{\F_0} (W_i)$ and let
$\widetilde{\Delta}_i \le \Out(W_i)$ be such that $K_i$ is a strongly $p$-embedded subgroup of $\widetilde{\Delta}_i$. For $\Delta_i$ the full preimage of $\tilde{\Delta}_i$ in $\Aut(W_i)$, write
$$\F =\langle \Mor(\F_0), \Delta_1,\ldots, \Delta_m \rangle.$$
Assume further that for each $1 \leq  i \leq  m$,
\begin{itemize}
\item[(a)]  $W_i$ is $\F_0$-centric  and minimal under inclusion amongst all $\F$-centric subgroups; and
\item[(b)] no proper subgroup of $W_i$ is $\F_0$-essential.
\end{itemize}
Then $\F$ is saturated.
\end{Thm}

\begin{proof}
See \cite[Theorem C]{S}.
\end{proof}

We now develop some tools for listing the possible $\F$-essential subgroups of a $p$-group $S$ when $\F$ is a saturated fusion system on $S$. We need two basic facts concerning the way in which a $p$-group acts on its subnormal subgroups.

\begin{Lem}\label{l:outinj}
Let $E$ be a finite $p$-group and $A \le\Aut(E)$. Suppose there exists a normal chain $$1=E_0 \unlhd E_1 \unlhd E_2 \unlhd \cdots \unlhd E_m=E$$ of subgroups such that for each $\alpha \in A$, $E_i\alpha=E_i$ for all $0 \le i \le m.$  If for all $1\le i \le m$, $A$ centralizes $E_i/E_{i-1}$, then $A \le O_p(\Aut(E))$.
\end{Lem}

\begin{proof}
See \cite[5.3.2]{GOR}.
\end{proof}

\begin{Lem}[Burnside]\label{l:burnside}
Let $S$ be a finite $p$-group. Then $C_{\Aut(S)}(S/\Phi(S))$ is a normal $p$-subgroup of $\Aut(S)$.
\end{Lem}

\begin{proof}
See \cite[5.1.4]{GOR}.
\end{proof}

The following result can also be found in  \cite[Lemma 3.4]{OV}.

\begin{Lem}\label{l:usecrit}
Let $S$ be a finite $p$-group and $F \le E \le S$  be such that $F$ is characteristic in $E$. If there exists $g \in N_S(E) \backslash E$ such that
\begin{itemize}
\item[(a)] $[g,E] \le F  \Phi(E)$ and
\item[(b)] $[g,F] \le \Phi(E)$,
\end{itemize}
then $E$ is not an $\F$-essential subgroup in any saturated fusion system $\F$ on $S$.
\end{Lem}

\begin{proof}
Since $C_{\Aut(E)}(E / \Phi(E)) \le O_p(\Aut(E))$ by Burnside's Lemma \ref{l:burnside} and since $F\Phi(E)$ is normal in $E$, it follows from Lemma \ref{l:outinj} that $c_g \in O_p(\Aut(E))$.  But then $O_p(\Aut(E)) \not \le \Inn(E)$  which means that $\Out_S(E) \cap O_p(\Out(E)) \neq 1$ and hence $E \notin \E_\F$ for any saturated fusion system $\F$ on $S$.
\end{proof}

We   need the next result about certain subgroups of $\mathrm{PGL}_3(p)$.

\begin{Prop}\label{p:l3pstr}
Let $p$ be an odd prime, and $G$ be a subgroup of $\mathrm{PGL}_3(p)$ which contains a strongly $p$-embedded subgroup. Then either $O^{p'}(G)$ is isomorphic to one of $\mathrm{PSL}_2(p)$ or $\SL_2(p)$ or $p=3$ and $G\cong \Frob(39)$.
\end{Prop}

\begin{proof}
See \cite[Proposition]{Valentina}.
\end{proof}

We end this section with a result about finite simple groups which will be required when proving that certain saturated fusion systems we construct are exotic. The next result is a special case of  \cite[Theorem]{Raul}. We use the following two facts about a Sylow $p$-subgroup $S$ of $\Gee_2(p)$: first $|S|=p^6$ and second if $K$ is an abelian normal subgroup of $S$, then $|K|\le p^3$ and $S/K$ is non-abelian (see Lemma~\ref{l:charz} (c)).

\begin{Thm}\label{t:simpfus}
Suppose that $p \ge 5$ and let $G$ be a finite simple group with a Sylow $p$-subgroup isomorphic to that of $\Gee_2(p)$. Then one of the following holds:
\begin{itemize}
\item[(a)] $p=5$ and $G \in \{\Gee_2(5),\BM,\HN, \Ly\}$;
\item[(b)] $p=7$ and $G  \in \{\Gee_2(7),\Mo\}$;
\item[(c)] $p > 7$ and $G  =\Gee_2(p)$.
\end{itemize}
\end{Thm}

\begin{proof} We use the classification of finite simple groups to prove this result. Assume that $G$ is a finite simple group with Sylow $p$-subgroup $S$ isomorphic to a Sylow $p$-subgroup of $\Gee_2(p)$ for $p \ge 5$.

 If $G$ is an alternating group $\Alt(n)$, then, as $S$  is non-abelian we require $n \ge p^2$. But then $|S| \ge p^{p+1}$. As $|S| = p^6$, we have $p=5$ and $S$ is isomorphic to the wreath product $5\wr 5 \in \Syl_5(\Alt(25))$. But then $S$ has an abelian subgroup of index $5$, a contradiction.

Suppose that $G$ is a Lie type group in characteristic $p$. Then, by \cite[Theorem 2.2.9]{GLS3}, $|S|= {p^a}^N$ where $N$ is the number of positive roots of the untwisted root system of $G$ and $p^a$ is the order of the centre of a long root subgroup of $G$. Since $|Z(S)|=p$, we have $N=6$.
The values of $N$ are given in \cite[Table 2.2]{GLS3} and this yields that the root systems with exactly $6$ positive roots are of type $\Gee_2$ and $\mathrm A_3$. Thus we need to consider the groups $\Gee_2(p)$, $\mathrm A_3(p)\cong \PSL_4(p)$ and ${}^2\mathrm A_3(p)\cong \PSU_4(p)$. In the latter two cases we see that a Sylow $p$-subgroup has an elementary abelian normal subgroup of order $p^4$, whereas in $S$ there is no such subgroup. Hence in this case we have $G \cong \Gee_2(p)$.

Suppose that $G$ is a Lie type group in characteristic $r \ne p$.  Then, by \cite[Theorem 4.10.2]{GLS3}, $S$ has a normal abelian subgroup $S_T$ such that $S/S_T$ is isomorphic to a subgroup of the Weyl group of $G$.  By Lemma~\ref{l:charz}, $S_T$ has order at most $p^3$ and  $S/S_T$ is non-abelian of order at least $p^3$. Now notice that, if a Weyl group $W$ has a non-abelian Sylow $p$-subgroup with $p \ge 5$, then $W$ has type $A_{n-1}$, $B_n$, $C_n$ or $D_n$ with $n \ge p$.  In particular, we see that $W$ has Sylow $p$-subgroups of order at least $p^{p+1}$.  Since $|S|= p^6$, we again have $p=5$ and   $S \cong 5\wr 5$, which is  a contradiction.

Finally assume that $G$ is a sporadic simple group.  Then, as $|S|= p^6$ and $p \ge 5$, using the orders of the sporadic simple groups \cite[Table 5.3]{GLS3} yields that $G$ must be   $\Ly$, $\HN$ or $\BM$ with $p=5$ or $\Mo$ with $p=7$.
\end{proof}

\section{Definition and basic properties of a Sylow $p$-subgroup of $\Gee_2(p)$ when $p \ge 5$}\label{s:constructs}

\subsection{Construction of $S$}\label{ss:constructs}
Let $q=p^f$ with $p \ge 5$ a prime  and $\mathbb{F}$ be a field of order $q$. In what follows, we construct a group $S$ which is isomorphic to a Sylow $p$-subgroup of $\Gee_2(q)$ (see the Appendix). To this end, we start with $V$ the 4-dimensional subspace of homogeneous polynomials of degree 3 in $\mathbb{F}[X,Y].$ Then $L=\mathbb{F}^\times \times \GL_2(\mathbb{F})$ acts on $V$ via  the $\mathbb F$-linear extension of  \[X^aY^b \cdot \left(t, \left( \begin{smallmatrix}
\alpha & \beta  \\
\gamma & \delta  \\ \end{smallmatrix} \right) \right)=t\cdot(\alpha X + \beta Y)^a \cdot (\gamma X + \delta Y)^b\] where $a+b=3$. We   define a bilinear function $\beta : V \times V \rightarrow \mathbb{F}$ by first defining $\beta$ on basis vectors by \[
 \beta(X^aY^b,X^cY^d)=\left\{\begin{array}{ll}
           0, & \textrm{ if $a \neq d$;} \\
           \frac{(-1)^a}{{3 \choose a}}, &  \textrm{ if $a = d$}
          \end{array}
          \right. \] and extending linearly. Let $Q$ be the group $(V \times \mathbb{F}^+,*)$ where $$(v,y)*(w,z)=(v+w,y+z+\beta(v,w)),$$ for $(v,y),(w,z) \in Q$. In \cite[Lemma 2.2]{PS} it is noted that $Q$ is a special group with the property that $$Z(Q)=\{(0,\lambda) \mid \lambda \in \mathbb{F}\}.$$

We now construct the group $S$ by extending the action of $L$ on $V$ to an action on $Q$ defined as follows: For $(t,A) \in L$ and $(v,z) \in Q$, \begin{eqnarray}\label{action}(v,z)^{(t,A)}&=&(v.{(t,A)},t^2(\det A)^3 z).\end{eqnarray} A simple check (carried out in the discussion before  \cite[Lemma 2.3]{PS}) shows that this action is a group action  (in the sense that $((v,y)(w,z))^{(t,A)}=(v,y)^{(t,A)}(w,z)^{(t,A)}$) and that the kernel of the action is \begin{eqnarray}\label{kernel}\left\{\left(\mu^{-3}, \left( \begin{smallmatrix}
\mu & 0    \\
0 & \mu  \\  \end{smallmatrix} \right)\right)\mid \mu \in \mathbb F^\times\right\}.\end{eqnarray} As in \cite{PS}, let
\begin{eqnarray*} B_0=\mathbb{F}^\times \times \left \{ \left( \begin{smallmatrix}
\alpha & 0  \\
\gamma & \beta  \\ \end{smallmatrix} \right)  \mid \alpha,\beta \in \mathbb{F}^\times, \gamma \in \mathbb{F} \right \}  \mbox{ and } S_0=\{1\} \times  \left\{ \left( \begin{smallmatrix}
1 & 0  \\
\gamma & 1  \\ \end{smallmatrix} \right) \mid  \gamma \in \mathbb{F} \right \} \end{eqnarray*} and set
\begin{eqnarray}\label{d:Bo}B=B_0Q \hspace{.5mm} \mbox{ and }  \hspace{.5mm} S=S_0Q. \end{eqnarray}
For $\lambda \in \mathbb{F}$, we define the following elements of $Q$: \begin{eqnarray*}x_6(\lambda)&=&(0,-2\lambda), \hspace{.5mm} x_5(\lambda)=(-\lambda X^3,0),\hspace{.5mm} x_4(\lambda)=(3\lambda X^2Y,0),\hspace{.5mm}\\ x_3(\lambda)&=&(-3\lambda XY^2,0), \hspace{.5mm} x_2(\lambda)=(\lambda Y^3,0).\end{eqnarray*}
Also write \[x_1(\lambda)=(1, \left( \begin{smallmatrix}
1 & 0  \\
\lambda & 1  \\ \end{smallmatrix} \right) ) \in S_0. \] Observe that $$S=\langle x_1(\lambda),x_2(\mu),x_3(\nu),x_4(\xi), x_5(o),x_6(\pi) \mid \lambda, \mu, \nu, \xi, o, \pi \in \mathbb{F} \rangle.$$

\subsection{Properties of $S$ and some  subgroups}
We now specialize to the case when $f=1$, so that $S$ is a Sylow $p$-subgroup of $\Gee_2(p)$. By the discussion in Section \ref{ss:constructs},  $S=\langle x_1,x_2,x_3,x_4,x_5,x_6\rangle$ where for each $1 \le j \le 6$ we write $x_j=x_j(1)$. Note that $S$ has nilpotency class $5$ and so $S$ is of maximal class. Thus let $$1 < Z=Z_1 < Z_2 < Z_3 < Z_4 < Z_5=S$$ be the upper (and lower) central series of $S$ where, for ease of notation, we set $Z= Z(S)$ and, for $ 2\le i \le 4$,  $Z_i = Z_i(S)$. Of particular importance to us will be the groups $$Q=\langle x_2,x_3,x_4,x_5,x_6 \rangle $$ and $$R= \langle x_1,x_3,x_4,x_5,x_6 \rangle.$$
From the construction of $S$, we see that

\begin{Lem}
The subgroup $Q$ is extraspecial of order $p^5$ and exponent $p$.\qed
\end{Lem}

In fact, if $p \ge 7$, then $S$ has exponent $p$ and, if $p=5$, then $S$ has exponent $25$. Indeed, $\Gee_2(p)$ has a $7$-dimensional faithful representation and so for $p \ge 7$, $S$ has exponent $p$. For $p=5$, we remark that every element of $S \setminus (R\cup Q)$ has order $25$ and $R$ and $Q$ both have exponent $5$.

\begin{Lem}\label{l:charz}
The following hold:
\begin{itemize}
\item[(a)] $Z=\langle x_6 \rangle$ and  $Z_2=\langle x_6,x_5 \rangle;$
\item[(b)] $R= C_S(Z_2);$
\item[(c)] $Z_3=\langle x_6,x_5,x_4 \rangle$ is elementary abelian and $$Z_4= C_{Q}(Z_2)= Q\cap R=\Phi(S)=\langle x_6,x_5,x_4,x_3 \rangle $$ and $Z_4$ is not abelian.
\item[(d)] $Q$ and $R$ are  characteristic maximal subgroups of $S$;
\item[(e)] the non-trivial normal subgroups of $S$ of order at most $p^4$ are the subgroups $Z_i$ with $1\le i \le 4$; and
\item [(f)] the action of $x_1$ on $Q/Z$ has a unique Jordan block.
\end{itemize}
\end{Lem}

\begin{proof}
Parts (a) and (b) follow directly from considering the description of  $S$. Since $\beta(X^3,X^2Y)=0$, we have $[x_4,x_5]=1$ so that $Z_3$ is abelian, and hence elementary abelian. Similarly, $x_3$, $x_4$ and $x_5$ all centralize $x_2$ so that $Z_4 \subseteq C_Q(Z_2)$. Since $\beta(X^3,Y^3) \neq 0$, $x_2 \notin C_Q(Z_2)$ so $Z_4 = C_Q(Z_2)$ and the remaining equalities in (c) are clear. To see that $Q$ is characteristic, we note that $Q/Z$ is the unique abelian subgroup of order $p^4$ in $S/Z$. That $R$ is characteristic follows from  the fact that $R=C_S(Z_2)$ and $Z_2$ is characteristic in $S$. Thus (d) is proved. Part (e) follows from the fact that $S$ has maximal class so  that the upper central series for $S$ and the lower central series for $S$ coincide. Part (f) follows from the fact that $S$ has maximal class.
\end{proof}

\begin{Lem}\label{l:z4char}
Suppose $X$ is a maximal subgroup of $S$ with $X \ne Q $. Then
\begin{enumerate}\item [(a)]$Z_3=\Phi(X)$ is characteristic in $X$;
\item [(b)]$Z_2 $ is characteristic in $X$; and
\item  [(c)]  either $Z_4$ is characteristic in $X$ or $X=R$.
\end{enumerate}
\end{Lem}

\begin{proof} As $X \ne Q$, we have  $S= QX$ and, as $Z_4 = \Phi(S)$, also $Z_4 < X$. Now note that, as $[Z_4,X] \ge [Z_4,Q]=Z$,  $$[Z_4,X]= [Z_4,QX]= [Z_4,S]= Z_3$$  and so $[X,X] \ge[Z_4,X] = Z_3$. Since $|X/Z_3|=p^2$, we deduce that $[X,X]= \Phi(X)=Z_3$.  In particular,  $Z_3$ is characteristic in $X$. This proves (a).

Now $$[Z_3,X]= [Z_3,QX]=[Z_3,S]=Z_2$$ and so (b) holds.
By Lemma~\ref{l:charz} (c) $Z_4$ centralizes $Z_2$. Let $\alpha \in \Aut(X)$ and assume that $Z_4\alpha \ne Z_4$. Then $X=Z_4 Z_4\alpha$. Since $Z_4\alpha$ centralizes $Z_2\alpha$, and  $Z_2\alpha=Z_2$ by (b),  we have $X\le R$ and as $X$ is maximal, $X=R$. This proves (c).

\end{proof}

As remarked in the introduction, we  need to prove that each of the fusion systems  $\F$ we construct is uniquely determined by the $\F$-automorphism groups of $\E_\F \cup \{S\}$. For this, a detailed description of the automorphism groups of $Q$, $R$ and $S$ is helpful.

\subsection{The structure of $\Aut(Q)$}

The structure of the automorphism group of an extraspecial $p$-group of exponent $p$ is well known, and we state it here only for convenience:

\begin{Prop}\label{l:autQ}
Set $A=\Aut(Q)$ and $\bar{A}=\Out(Q)$. There exists $\theta \in A$ of order $p-1$ such that $A=\langle \theta \rangle \cdot C_A(Z(Q)) \mbox{ and } \langle \theta \rangle \cap C_A(Z(Q)) =1.$ Moreover  $\overline{C_A(Z(Q))} \cong \Sp_4(p)$ and   $\bar{A} \cong \GSp_4(p)$.
\end{Prop}

\begin{proof}
See \cite[Theorem 1]{Winter}.
\end{proof}

\subsection{The structure of $\Aut(R)$} The next lemma provides us with a rather precise  description of $\Aut(R)$.

\begin{Lem}\label{l:autR} Let $A= \Aut(R)$, $\overline A = \Out(R)$ and put $$\mathcal A= \{(x,y) \in R\times R\mid R= \langle x,y\rangle\}.$$ Then the following  hold:
\begin{itemize}
\item[(a)] $\Inn(R) \cong p^{1+2}_+$;
\item[(b)] $|\Aut(R)|= p^7(p^2-1)(p-1)$;
\item[(c)]  if $(x,y), (x_1,y_1) \in \mathcal A$, then there exists $\theta \in \Aut(R)$ such that $x\theta = x_1$ and $y \theta = y_1$;
\item[(d)] $A/O_p(A) \cong \GL_2(p)$;
\item[(e)] $O_p(\overline A)$ is elementary abelian of order $p^3$ and as an $O^{p'}(A/O_p(A))$-module is isomorphic to the module of $2\times 2$-matrices over $\mathbb F_p$ of trace $0$ acted upon by conjugation by $\SL_2(p)$;
\item[(f)] $Z(\overline  A) $ has order $2$; and
\item[(g)] there exists a subgroup $\overline X$ of $\overline A$ with $\overline X \cong \GL_2(p)$.
\end{itemize}
\end{Lem}

\begin{proof}
Since $S$ has maximal class and $Q \cap R= Z_4$, $[R, Q\cap R]= [S,Z_4]= Z_3$ and so $\Inn(R) \cong R/Z_2 $ is extraspecial. Since  $(Q\cap R)/Z_2$ is elementary abelian and $x_1 Z_2$ has order $p$, we see that $\Inn(R)$ has exponent $p$.  This proves (a).

By \cite[Lemma 5.2] {PaR},  the map
$$\widetilde{}: C_{\Aut(R)}(R/Z(R)) \rightarrow \Hom(R,Z(R))$$ which sends $\Psi \in C_{\Aut(R)}(R/Z(R))$ to the map $\widetilde \Psi \in \Hom(R,Z(R))$ defined, for $g \in R$,  by  $g\widetilde{\Psi}=g^{-1} (g\Psi) $ is an isomorphism.  Moreover, $\;\widetilde{}\;$ is $\Aut(R)$-invariant. Indeed suppose that $\alpha \in \Aut(R)$ and $g \in R$.   Then, for $\theta \in \Hom(R, Z(R))$, we have  \begin{eqnarray}\label{55} g\theta ^\alpha= g\alpha^{-1}\theta\alpha\end{eqnarray} and so we calculate
 $$g\widetilde {\Psi^\alpha} = g^{-1} (g \Psi^\alpha)=g^{-1}(g\alpha^{-1}\Psi \alpha  ) = (g^{-1}\alpha^{-1}(g\alpha^{-1}\Psi))\alpha= ((g\alpha^{-1})\widetilde \Psi)\alpha=g\widetilde {\Psi}^\alpha.$$
Since $\Hom(R,Z(R)) \cong \Hom(R/\Phi(R), Z(R))$ we see that $C_{\Aut(R)}(R/Z(R))$ is isomorphic to the set of all linear transformations from a $2$-space to a $2$-space. Thus $C_{\Aut(R)}(R/Z(R))$ is elementary abelian of order $p^4$.

Next we collect some automorphisms of $R$ which can be obtained from a parabolic subgroup $P$ in  $G=\Gee_2(p)$. After identifying $S$ with a Sylow $p$-subgroup of $G$, the relevant parabolic subgroup is $P= N_G(R)$ and there we observe $P/C_G(R) =\Aut_G(R)\cong p^{1+2}_+{:}\GL_2(p)$.

Also  $$\Aut_G(R) \cap C_{\Aut(R)}(R/Z(R)) = \Inn(R)\cap C_{\Aut(R)}(R/Z(R)) = Z(\Inn(R))=\Phi(R)/Z(R)= Z_3/Z_2$$ has order $p$. Hence $\Aut(R)$ has order at least $$|C_{\Aut(R)}(R/Z(R))\Aut_G(R)|=\frac{p^4.p^3. p.(p^2-1)(p-1)}{p}= p^7(p^2-1)(p-1).$$

We now establish an upper bound for $|\Aut(R)|$ and thus simultaneously prove parts (b) and (c). Since $R/\Phi(R)= R/Z_3$ has order $p^2$, there exist $x,y \in R$ such that $R= \langle x, y\rangle$.  We count  the possible  number of images of $x$ and $y$ under an  automorphism $\theta$ of $R$.  Plainly, $R= \langle x \theta , y\theta\rangle$, $x\theta \not \in \Phi(R)$ and $y\theta  \not \in \langle x \theta \rangle \Phi(R)$.  There are at most $$|R|-|\Phi(R)|=p^5-p^3$$ choices for  $x\theta $ and then $$|R|- |\langle x \theta \rangle \Phi(R)|  =p^5-p^4$$ choices for $y\theta$.  Thus there are at most $(p^5-p^3)(p^5-p^4)= p^7(p^2-1)(p-1)$ automorphisms of $R$. Thus (b) and (c) hold.

Furthermore, from the discussion   in the proof of (b), we see that $\overline A$ contains $\overline{\Aut_G(R)} \cong \GL_2(p)$ (which gives (d) and (g)) and $\Inn(R) \cap  C_{\Aut(R)}(R/Z(R))$ has order $p$.  Hence $O_p(A)=\Inn(R)  C_{\Aut(R)}(R/Z(R))$ and $$O_p(\overline A) \cong C_{\Aut(R)}(R/Z(R))/(\Inn(R) \cap  C_{\Aut(R)}(R/Z(R)))$$ and this isomorphism is as $A$-groups. In particular, $O_p(\overline A)$ is elementary abelian of order $p^3$. Since $C_{\Aut(R)}(R/Z(R))$ is isomorphic to the set of all linear transformations from a $2$-space to a $2$-space and is also an $A$-group, we infer that as an $O^{p'}(A/O_p(A))$-module, $O_p(\overline A)$ is isomorphic to the module of trace zero $2\times 2$-matrices over $\mathbb F_p$ with $\SL_2(p)$ acting by conjugation.  This proves (e).

Because $O_p(\overline A)$ is a minimal normal subgroup of $\overline A$ by (e), $Z(\bar{A}) \cap O_p(\bar{A}) = 1$ so $$Z(\overline A) \cong Z(\overline A) O_p(\overline A)/O_p(\overline A) \le Z( \overline A /O_p(\overline A)) \cong Z(\GL_2(p))$$ by (d).
Thus we need to determine the centre of the preimage of $Z( \overline A /O_p(\overline A)) $.  Since $O_p(\overline A)$ is abelian, it suffices to determine which elements of $Z( \overline A /O_p(\overline A)) $ lift to elements of $A$ which  centralize $O_p(\overline A)$.

Let $(x,y) \in \mathcal A$ with $y \in Q$.  Then, by (c), the map $x\mapsto x^a$, $y \mapsto y^a$ for $1 \le a \le p-1$ extends to a unique automorphism $\theta$ of $R$ and $\theta \in Z(\overline A /O_p(\overline A))$.
 Then define $l= [x,y]$, $m=[l,x]$ and $n=[l,y]$.
Notice that, as $R/Z_2\cong \Inn(R)$ is extraspecial, $l \in Z_3 \setminus Z_2$. Because $C_{Q}(Z_3)= Z_3$, we then see that $1\ne n \in Z$.  As $x$ acts on $Q/Z$ with a single Jordan block, we have $C_{Q/Z}(x)= Z_2/Z$ and so $m \in Z_2 \setminus Z$. This shows that $Z_2= \langle m, n\rangle$.
Now we use \cite [Lemmas 2.2.1 and 2.2.2]{GOR} to notice first that
$$l\theta=[x\theta, y \theta]=[x^a,y^a] = [x,y]^{a^2}z $$ for some $z \in Z_2$ and then  calculate that $$m\theta=[\ell\theta,x\theta]=[[x,y]^{a^2}z,x^a]=[[x,y]^{a^2},x^a]=[[x,y],x]^{a^3}=m^{a^3}$$
where the third equality follows from \cite[Theorem 2.2.1]{GOR}.
Similarly, we determine \begin{eqnarray}\label{interestingfact} n\theta=[[x,y]^{a^2}z,y^a]=[[x,y]^{a^2},y^a]=[[x,y],y]^{a^3}=n^{a^3}.\end{eqnarray}
Now using Equation \ref{55} and noting that $\theta$ operates as  the scalar $a$ on $R/\Phi (R)$ and $a^3$ or $Z_2$, we calculate, for $\Psi \in C_{\Aut(R)}(R/Z(R))$ and $g \in R/\Phi(R)$,
$$g \widetilde \Psi^\theta = g \theta ^{-1} \widetilde \Psi \theta = g^{a^{-1}} \widetilde \Psi \theta =
 ((g \widetilde \Psi) ^{a^{-1}})\theta =  (g \widetilde \Psi) ^{a^{-1}a^3}= (g\widetilde \Psi)^{a^2}$$
 Thus we see that $\theta$ centralizes $ C_{\Aut(R)}(R/Z(R))$ if and only if $ (g\widetilde \Psi)^{a^2}=  g\widetilde \Psi$ for all $g \in R/\Phi(R)$ and $\Psi \in  C_{\Aut(R)}(R/Z(R))$ which is if and only if  $a^2=1$. As $\theta$ induces a scalar action on $C_{\Aut(R)}(R/Z(R))$ and $O_p(\overline A)=   C_{\Aut(R)}(R/Z(R))/(\Inn(R) \cap  C_{\Aut(R)}(R/Z(R)))$,
 we now deduce that $C_{\overline A}(O_p(\overline A))$ has order $2p^3$ and part (f) follows from this.
\end{proof}

It is perhaps interesting to note that Equation~\ref{interestingfact}  implies that $\Aut(R)/C_{\Aut(R)}(R/\Phi(R))\cong \GL_2(p)$ whereas $\Aut(R)/C_{\Aut(R)}(Z(R)) \cong \GL_2(p)/X$ where $X$ is central of order $(p-1,3)$.

\begin{Lem}\label{l:runique} Let $\overline A= \Out(R)$ and suppose that $\overline Y \le \overline A$, $\overline T \in \Syl_p(\overline Y)$ with $|\overline T|=p$ and $|C_{\overline Y}(\overline T)| > 2$.  Assume that  $\overline X \le \overline A$ with $\overline X \cong \SL_2(p)$ and $\overline Y \le N_{\overline A}(\overline X)$.  Then $\overline{XY} \le C_{\overline A}(C_{\overline Y}(\overline T))\cong \GL_2(p)$. In particular, if such an $\overline X$ exists, then it is uniquely determined by $\overline Y$.
\end{Lem}

\begin{proof} By Lemma~\ref{l:autR} (d) and (e), $\overline A$ has shape $p^{3}{:}\GL_2(p)$ and $\overline U=O_p(\overline A)$ is a minimal normal subgroup of $\overline {UX}$.  Since $\overline Y$ normalizes $\overline X$ and $N_{\overline U}(\overline X) = 1$, we deduce that $\overline {XY} $ is isomorphic to a subgroup of $\overline A /\overline U \cong \GL_2(p)$. In particular, $|C_{\overline Y}(\overline T)|\le Z(\overline {XY})$ from the structure of $\GL_2(p)$. Now, as $C_{\overline Y}(\overline T)$ has order greater than $2$, Lemma~\ref{l:autR} (f) and the fact that $\overline U$ is a minimal normal subgroup of $\overline {UX}$ imply that $C_{\overline Y}(\overline T) \cap \overline U=1$ (note that  $C_{\overline Y}(\overline T) \cap \overline U$ is normalized by $\overline X$.) Thus  $C_{\overline A}(C_{\overline Y}(\overline T))\cong \GL_2(p)$ and this proves the result.
\end{proof}
\subsection{The structure of $\Aut(S)$}
We conclude this section with description of $\Aut(S).$
\begin{Lem} \label{l.CA}
Suppose that $X$ is a group and $Y$ is a normal subgroup of $X$ of index $p$ where $p$ is a prime.  Then $[X,C_{\Aut(X)}(Y)] \le C_X(Y)$.
\end{Lem}

\begin{proof}
Select $x \in X \setminus Y$ and notice that since $p$ is prime every element $z$ of $X$ can be written as $z=yx^i$ for some $y \in Y$ and $1 \le i \le p$.  Now for each $\alpha \in C_{\Aut(X)}(Y)$, $$[z,\alpha] =[yx^i, \alpha]= [y,\alpha]^{x^i}[x^{i},\alpha] = [x^{i}, \alpha]$$ and so it suffices to show that $[x^i,\alpha] \in C_X(Y)$.

Let $y_1 \in Y$.  Then $(x^{i})^{y_1}= y_2x^{i}$ for $y_2=[y_1,x^{-i}]$ and so
\begin{eqnarray*}[x^i,\alpha] ^{y_1}&=&(x^{-i}(x^i)\alpha)^{y_1}= (x^{-i})^{y_1} ((x^i)\alpha^{(y_1)\alpha})= (x^{-i})^{y_1}((x^i)^{y_1})\alpha\\&=& (y_2x^{i})^{-1} (y_2x^{i})\alpha=x^{-i}y_2^{-1} y_2(x^i)\alpha = [x^{i},\alpha] \end{eqnarray*}
Hence $[x^i,\alpha] \in C_X(Y)$ and consequently $[X,\alpha] \le C_X(Y)$. The result follows.
\end{proof}

\begin{Lem}\label{l:liftnorm}
$\Aut(S)/C_{\Aut(S)}(S/\Phi(S))$ is isomorphic to the subgroup of diagonal matrices in $\GL_2(p)$. In particular, $|\Aut(S)|=p^a(p-1)^2$ for some natural number $a$. Furthermore, $\Aut(S) =\Aut_B(S)C_{\Aut(S)}(S/\Phi(S))$.
\end{Lem}

\begin{proof}
By Lemma \ref{l:burnside}, $C_{\Aut(S)}(S/\Phi(S))$ is a $p$-group. Taking $B$ as defined in Equation \ref{d:Bo}, using Equation~\ref{kernel} we obtain that the image of $\Aut_{B}(S)$ in $\Aut(S)/C_{\Aut(S)}(S/\Phi(S))$ is isomorphic to $(p-1) \times (p-1)$.

As $S/\Phi(S)$ is elementary abelian of order $p^2$, we know that $\Aut(S)/C_{\Aut(S)}(S/\Phi(S))$ is isomorphic to a subgroup of $\GL_2(p)$. By Lemma~\ref{l:charz} (d), $Q$ and $R$ are characteristic in $S$. Thus $\Aut(S)/C_{\Aut(S)}(S/\Phi(S))$  is isomorphic to a subgroup of the diagonal matrices in $\GL_2(p)$. This proves the main claim.
\end{proof}

\section{Candidates for the essential subgroups when $p \ge 5$}\label{s:candidates}

Suppose that $p \ge 5$ and let $S$ be the $p$-group defined in Section 3 and adopt all the notation introduced there.  We   require the following additional piece of notation:

\begin{Not}\label{n:wx}
Define:
$$\begin{array}{cll} W_x&= \langle Z , x\rangle & x \in S \setminus (Q\cup R); \text{ and } \\
U_x&= \langle Z_2, x\rangle & x \in S \setminus Q.
\end{array}$$
Also put
$$\W =\{W_x\mid x \in S \backslash (Q\cup R)\} \hspace{2mm} \mbox{ and } \hspace{2mm}
\U =   \{U_x\mid x \in S \backslash Q\}.$$
\end{Not}

The goal of this section is to prove the following result:

\begin{Thm}\label{t:esslist}
Let $\F$ be a saturated fusion system on $S$ and denote by $\E$ the set of $\F$-essential subgroups. Then $$\E \subseteq \{Q,R\} \cup \W. $$  Moreover, if $\W \cap \E\ne \emptyset$,  then $p=7$.
\end{Thm} Thus our hypotheses are that $\F$ is a saturated fusion system on $S$ with $O_p(\F)=1$ and $\E=\E_\F$ is the set of $\F$-essential subgroups of $S$. The proof of Theorem \ref{t:esslist} will proceed in a series of steps.

\begin{Lem}\label{l:econtq}
If $E \le Q$ is $\F$-essential, then $E=Q$.
\end{Lem}

\begin{proof} Suppose that $E$ is $\F$-essential with $E \le Q$ but that $E \ne Q$.  Then $N_Q(E) > E$ and  $[E,N_Q(E)] \le Q'= \Phi(Q)$. If $\Phi(E) \ne 1$, then we have $\Phi(E)=\Phi(Q)$,  $[E,N_Q(E)] \le \Phi(E)$ and $[\Phi(E), N_Q(E)]=1$. Thus Lemma~\ref{l:usecrit} implies that $E$ is not essential, a contradiction. Therefore $\Phi(E)=1$ and $E$ is elementary abelian.  Since $E \ge C_Q(E)$, we deduce that $E$ is a maximal abelian subgroup of $Q$. Hence $|E|= p^3$. Now $Q/E$ embeds into $\Aut_\F(E)$ and so Proposition~\ref{p:l3pstr} provides a contradiction as $|Q/E|= p^2$. Hence, if $E \le Q$ and $E \in \E$, then $E=Q$.

\end{proof}

\begin{Lem}\label{l:2cases}
If $E \nleq Q$ is $\F$-essential, then either $E=R$ or $E \in \U \cup \W$.
\end{Lem}

\begin{proof}
Since $E$ is $\F$-centric, $Z \le C_S(E) \le E$, so we may assume that $|E|=p^t$ for some $2 \le t \le 5$. If $t=2$ then $E=Z\langle x \rangle$ and as $E$ must be centric, $C_S(E) \not \ge Z_2$.  Hence, as $E \not \le Q$, we have $x \in S \backslash ({Q\cup R})$. Thus $E \in \mathcal W$ in this case.

Suppose that  $t \geq 3$. Then $Z \le  Q \cap E$ and, as $Q/Z$ is abelian, we have  $Q \cap E \unlhd Q$. As $Q$ is normal in $S$, $Q\cap E$ is normal in $E$ and so $Q\cap E$ is normal in $S=\langle Q,E\rangle$. Therefore, by Lemma \ref{l:charz}(e) we have that $E \cap Q=Z_{t-1}$.

If $t=3$ (so that $E \cap Q= Z_2$), then  $E \in \mathcal U$.

It remains to show that if $t > 3$ then $E=R$.
Suppose that  $t=4$. Then $E = \langle Z_3, x\rangle$ for some $x \in S \backslash Q$.  We have
 $$Z_2=[E \cap Q,S]=[E \cap Q, EQ]=[E \cap Q,E] \le [E,E]$$  and so we infer that $Z_2=[E,E]=\Phi(E)$. By Lemma \ref{l:charz}(c) $Z_3$  is elementary abelian.  If $Z_3$ is normalized by $\Aut_\F(E)$, then using Lemma~\ref{l:outinj} together with $[E,N_S(E)] \le Z_3$ and $[Z_3,N_S(E)] \le Z_2 = \Phi(E)$  yields that $\Aut_S(E) \le O_p(\Aut_\F(E))$ contrary to $E$ being $\F$-essential. Hence there exists  $\alpha \in \Aut_\F(E)$ such that $Z_3 \ne Z_3\alpha$.  As $|E| = p^4$, we have $E= Z_3Z_3\alpha$ and $Z_3\cap Z_3\alpha = Z_2$.  Since $Z_3$ is elementary abelian, this means that $Z_2= \Phi(E) = Z(E)$ and we remark that this group is elementary abelian. Let $x \in Z_3\setminus Z(E)$ and $y \in Z_3\alpha \setminus Z(E)$. Then $E= \langle x,y\rangle$ and $x$ and $y$ have order $p$.  Set $N= \langle[ x,y]\rangle$.  Then $N \le \Phi(E)=Z(E)$ and $N$ has order $p$.  But then $E/N$ is generated by $xN$ and $yN$ and these elements commute and have order $p$.  It follows that $E/N$ has order both $p^3$ and $p^2$, a contradiction.

Finally, suppose that $t=5$ so $E$ is a maximal subgroup of $S$ which is not equal to $Q$. If $E \neq R$, we claim that the hypotheses of Lemma \ref{l:usecrit} are satisfied with $F=Z_4$. Indeed, $Z_4$ is characteristic in $E$ by Lemma \ref{l:z4char} (c). Moreover $\Phi(E)=Z_3$, so that for any $x \in S \backslash E$, $$[x,E] \le \Phi(S)=Z_4 \mbox{ and } [x,F]=[x,Z_4]\le Z_3=\Phi(E).$$ Hence by Lemma \ref{l:usecrit} $E$ is not $\F$-essential. Thus, if $E\ne Q$ is an $\F$-essential subgroup of order $p^5$, then $E =R$ and this completes the proof.
\end{proof}

We also observe the following fact which can also be deduced from the remark after Lemma~\ref{l:autR}.

\begin{Lem}\label{l:AFR1}
If $R \in \E$, then $\Out_\F(R)$ is isomorphic to a subgroup of $\GL_2(p)$ and $O^{p'}(\Out_\F(R)) \cong \SL_2(p)$. Furthermore, $O^{p'}(\Out_\F(R))$ acts faithfully on $R/\Phi(R)$ and on $Z_2=Z(R)$.
\end{Lem}

\begin{proof} By Lemma~\ref{l:burnside},  $\Out_\F(R)$ acts faithfully on $R/\Phi(R)$ which is elementary abelian of order $p^2$.  Since $R \in \E$ and any two distinct cyclic subgroups of order $p$ in $\GL_2(p)$ generate $\SL_2(p)$, the main statement follows and the action of $O^{p'}(\Out_\F(R))$ on $R/\Phi(R)$ is of course faithful. To see that the action on $Z(R)$ is faithful, it suffices to show that the central involution $t$ of $O^{p'}(\Out_\F(R))$ acts non-trivially on $Z(R)$.  Notice that $O^{p'}(\Out_\F(R))$ centralizes the cyclic group $Z_3/Z_2$. Hence, as $t$ inverts $Z_4/Z_3$, $t$ also inverts $[Z_4,Z_3]= Z$. This proves the claim.
\end{proof}

\begin{Lem}\label{l:outfsid}
 $\Out_\F(S)$  is conjugate in $\Out(S)$ to a subgroup of diagonal matrices in $\GL_2(p)$. In particular we may assume that $\Aut_\F(S) \le \Aut_B(S),$ where $B$ is as defined in Equation \ref{d:Bo}. Moreover an element $$d= \left(t, \left( \begin{smallmatrix}
\lambda & 0    \\
0 &   1\\  \end{smallmatrix} \right)\right)\in B$$ with $t, \lambda \in \mathbb F_p^\times$  centralizes $Z$ if  and only if $t^2\lambda^3=1$.
\end{Lem}

\begin{proof} By Lemma~\ref{l:liftnorm} and Hall's Theorem, $\Out_\F(S)$ is $\Out(S)$-conjugate to a subgroup of $\Out_B(S)$. An explicit calculation using Equations \ref{action}  and \ref{kernel}  gives the second part of the result.  \end{proof}

\begin{Lem}\label{l:diags} Suppose that $D \le \Out_\F(S)$  normalizes a non-trivial proper subgroup of $S/Z_4$ which is not equal to $Q/Z_4$ or $R/Z_4$.  If $D$ centralizes $Z$, then $D$ has order dividing $5$.
\end{Lem}

\begin{proof} Let $c_d \in D^\#$. Then, by Lemma~\ref{l:outfsid},  $d= \left(t, \left( \begin{smallmatrix}
\lambda & 0    \\
0 & 1  \\  \end{smallmatrix} \right)\right)$ with $t^2\lambda^3=1$.
We calculate that on $Q/Z_4$ (which is generated by $Y^3$) $c_d$ acts by scaling $Q/Z_4$ by $t$, and on $R/Z_4$ we calculate $c_d$ scales by $\lambda$.
Thus for a diagonal subgroup to remain fixed by $d$, we require $t=\lambda$. On the other hand, from Lemma~\ref{l:outfsid} we know $t^2\lambda^3= \lambda^5=1$.
It follows that   $D$ is cyclic of order dividing $5$.
\end{proof}

We use Lemmas \ref{l:outfsid} and \ref{l:diags} to help eliminate the possibility that $\F$ contains an essential subgroup in $\mathcal U$. We achieve this in the next three lemmas.

\begin{Lem}\label{l:nop3}
If $U_x \in \E$ for some $x \in S \backslash Q$, then $U_x$ is abelian (equivalently $U_x \le R$).
\end{Lem}

\begin{proof}
Write $E=U_x$ for some $x \in S \backslash Q$. If $E$ is non-abelian, then  $x \notin R$ and, as $E$ is $\F$-essential, $E \cong p_+^{1+2}$ with $[E,E]=Z$. Since $\Out_\F(E)$ acts faithfully on $E/\Phi(E)$, we have $\Out_\F(E)$ is isomorphic to a subgroup of $\GL_2(p)$ containing $\SL_2(p)$ just as in Lemma~\ref{l:AFR1}. Let $C= C_{\Aut_\F(E)}(Z)$.  Then $C/\Inn(E)\cong \SL_2(p)$ and $N_C(\Aut_S(E))$ is cyclic of order $p-1$.
Since $\F$ is saturated, the elements of  $C$ extend to a maps in $\Aut_\F(N_S(E))$. Now using  $x \notin R$, we see that each $\alpha \in C$ is the restriction of an element $\bar{\alpha} \in \Aut_\F(S)$ by Lemmas \ref{l:econtq} and \ref{l:2cases}. But then $\Out_\F(S)$ contains a subgroup $C_0$ of order $p-1$ which centralizes $Z$ and whose elements restrict to elements in $C$. We have that $C_0$ normalizes $Z_4E$ and so, as $p-1$ does not divide $5$, Lemma \ref{l:diags} implies that $EZ_4= R$ or $EZ_4=Q$. Since $x \not \in Q \cup R$,  we have a contradiction.
\end{proof}

We have the following observation:

\begin{Lem}\label{l:UxEthenRnotE}
Suppose that $U_x \in \E$. Then $R \not \in \E$.  In particular, $O_p(\Aut_\F(R)) =\Aut_S(R)=S/Z_2$.
\end{Lem}

\begin{proof} Suppose that $R \in \E$. Then Lemma \ref{l:AFR1} implies that $O^{p'}(\Out_\F(R)) \cong \SL_2(p)$. By Lemma~\ref{l:nop3}, $U_x \le R$ and $U_x\Phi(R)$ is a maximal subgroup of $R$.  Since $O^{p'}(\Aut_\F(R)) \cong \SL_2(p)$ acts transitively on the maximal subgroups of $R$, we see that $U_x\Phi(R)$ is conjugate to $Q\cap R$ by some $\alpha \in O^{p'}(\Aut_\F(R))$.  Thus $U_0=U_x\alpha \ge Z_2$
  and $N_S(U_0) \ge Q$.  Since $N_S(U_x)= U_x\Phi(R)$ by Lemma \ref{l:autR}(a), we see that $U_x$ is not fully $\F$-normalized and thus $U_x \not \in \E$.  This proves the claim. \end{proof}

\begin{Lem}\label{l:uxstruct}
Suppose $U_x \in \E$ for some $x \in S \setminus Q$. Then $O^{p'}(\Aut_\F(U_x))\cong \SL_2(p)$ and the following statements hold:
\begin{itemize}
\item[(a)] As an $\mathbb{F}_pO^{p'}(\Aut_\F(U_x))$-module, $U_x$ is the direct sum of a $2$-dimensional module and a $1$-dimensional trivial module.
\item[(b)]$Z= C_{U_x}(O^{p'}(\Aut_\F(U_x)))$.
\end{itemize}
\end{Lem}

\begin{proof}
By Lemma \ref{l:nop3}, $U_x$ is elementary abelian and so, as $U_x$ is centric we may regard $U_x$ as a faithful $\mathbb{F}_p\Aut_\F(U_x)$-module. In particular,  $\Aut_\F(U_x)$ is isomorphic to a subgroup of $\GL_3(p)$. Since $\Aut_\F(U_x)$ has a strongly $p$-embedded subgroup and $p \ge 5$, Proposition \ref{p:l3pstr} yields
 $O^{p'}(\Aut_\F(U_x))$ is isomorphic to either of   $\mathrm{PSL}_2(p)$ or $\SL_2(p)$.  Now  $\Aut_S(U_x) = \Aut_{Z_3}(U_x)$ and $Z_3$ is abelian by Lemma~\ref{l:charz}, we have
 $$[U_x,\Aut_S(U_x) ,\Aut_S(U_x) ]=[U_x,Z_3,Z_3]\le [Z_3,Z_3]=1.$$
 In particular, as a subgroup of $\GL_3(p)$, the Jordan form  of the elements  of $\Aut_S(U_x) $ have one block of size $2$ and a trivial block. It follows that $O^{p'}(\Aut_\F(U_x))\cong \SL_2(p)$ because the $p$-elements of  $\mathrm{PSL}_2(p)$ have Jordan block of size $3$. Let $\tau \in O^{p'}(\Aut_\F(U_x))$ be an involution contained  in the centre of $\SL_2(p)$.  Then $\tau \in Z(\Aut_\F(U_x) )$ and $U_x=[U_x,\tau]\oplus C_{U_x}(\tau)$ is an $O^{p'}(\Aut_\F(U_x) )$ decomposition of $U_x$ as the direct sum of a $2$-dimensional module and a $1$-dimensional trivial module.

We now prove (b).  We have
 $[U_x,\Aut_S(U_x) ]=[U_x,Z_3]\le [S,Z_3]=Z_2$ and, by (a), $[U_x,\Aut_S(U_x) ]$ has order $p$.   Since $\Aut_S(U_x) = \Aut_{Z_3}(U_x)$  and $Z_3$ is abelian, we have $C_{U_x}(\Aut_S(U_x))= C_{U_x}(Z_3) = Z_2$.
 If $[U_x,\Aut_S(U_x) ] \le Z$, then $[S,Z_3] = [U_xQ,Z_3] \le Z$, which is impossible. Thus $[U_x,\Aut_S(U_x) ] $ is a subgroup of $Z_2$ of order $p$ which is not contained in $Z$. Let $\tau \in Z(\Aut_\F(U_x))$ have order $2$.  Then $[U_x,\Aut_S(U_x) ]$ is inverted by $\tau$ and to prove the result it suffices to show that $Z$ is normalized by  $\tau$  for then $Z_2= [U_x,\Aut_S(U_x) ] Z$ with $Z$ centralized by $\tau$.  Suppose that $\tau$ does not normalize $Z$.  Since $\tau$ normalizes $\Aut_S(U_x)$ and $\F$ is saturated, $\tau$ lifts to $\bar \tau \in \Aut_\F(N_S(U_x))$.  Now using Theorem~\ref{t:alp} and Lemmas \ref{l:2cases} and \ref{l:UxEthenRnotE}, we see that $\bar \tau$ is the restriction of some $\tau^* \in \Aut_\F(S)$.  But then $$Z\tau= Z\bar \tau = Z\tau^*= Z,$$ which is a contradiction. This proves (b).
\end{proof}

\begin{Lem}\label{l:Nop3}
$U_x \notin \E$ for all $x \in S \setminus Q$.
\end{Lem}

\begin{proof} We have $O^{p'}(\Aut_\F(U_x)) \cong \SL_2(p)$ by Lemma~\ref{l:uxstruct}.
Let $t \in Z(O^{p'}(\Aut_\F(U_x)))$ be an involution. Then $t \in N_{\Aut_\F(U_x)}(\Aut_S(U_x))$ and so $t=\hat{\tau}|_{U_x}$ for some $\hat \tau \in \Aut_\F(N_S(U_x))$. Since $R \notin \E$ and $U_x \not \le Q$,  $\hat{\tau}$ must extend to a map ${\tau} \in \Aut_\F(S)$ by Lemmas \ref{l:econtq} and \ref{l:2cases}.  Now, by Lemma \ref{l:uxstruct} (a) and (b),  $\tau$ centralizes $Z$, inverts $Z_2/Z$ and inverts $R/Z_4= U_xZ_4/Z_4 \cong U_x/Z_2$. Furthermore, as $Q$ is characteristic in $S$, $\tau$ acts on $Q/Z_4$. Since $\tau$ has even order, $\tau$ does not invert $Q/Z_4$ by Lemma~\ref{l:diags}.  Let $a\in Q \setminus Z_4$ and $b \in Z_2 \setminus Z$ with $b\tau =b^{-1}$. Then, for some $1 \le e \le p-1 $, $aZ_4\tau = a^eZ_4$ and so, as
$[a,b]\in Z^\#$, we obtain  $$[a,b]=[aZ_4,b]\tau =[a Z_4\tau, b^{-1}]= [a^e,b^{-1}]= [a,b]^{-e}.$$ Hence $e=p-1$. As we have argued that $Q/Z_4$ is not inverted by $\tau$, this is a contradiction.
\end{proof}

\begin{Lem}\label{l:wine}
Suppose that $W \in \E \cap \W$. Then $p=7$ and the following hold:\begin{enumerate}
\item[(a)] $\Aut_\F(W) \cong \SL_2(7)$ is uniquely determined;
\item [(b)] $|N_{\Aut_\F(S)}(W)\Inn(S)/\Inn(S)| = 6$;
\item[(c)] there exists $\theta \in N_{\Aut_\F(Q)}(\Aut_S(Q))$  such that $\theta$   induces an automorphism of order $6$ on both $\Out_{S}(Q)$ and $Z$.
\end{enumerate}
Furthermore, the subgroup $N_{\Aut_\F(S)}(W)\Inn(S)/\Inn(S) \le \Out(S)$ is  generated by the images of  $c_d$ where
$d=\left(\lambda,\left(\begin{smallmatrix} \lambda &0\\0&1\end{smallmatrix}\right)\right)\in B$ with $\lambda\in \mathbb F_7^\times$ which acts as scalars on $S/\Phi(S)$ and is independent of the choice of $W \in \E \cap \W$.
\end{Lem}

\begin{proof}  Suppose that $W \in\E\cap  \W$. Then $O^{p'}(\Aut_\F(W))\cong \SL_2(p)$. Since $\F$ is saturated, any element of $N_{\Aut_\F(W)}(\Aut_S(W))$ extends to an automorphism of $N_S(W)$ and then, by Theorem~\ref{t:alp}, to an automorphism of $S$.  Let $\delta \in O^{p'}(\Aut_\F(W))$ normalize $\Aut_S(W)$ have order $p-1$. Then we may assume that $\delta$ scales $Z$ by $\lambda^{-1}$ and $W/Z$ by $\lambda$ where $\lambda$ is a generator of $\mathbb F_p^\times$. Let $\delta^* \in \Aut_\F(S)$ extend $\delta$.    By Lemma~\ref{l:outfsid}, we may suppose that $\delta^*$ acts as $c_d$ where we may assume  $d=\left(t, \left( \begin{smallmatrix}
\lambda & 0    \\
0 & 1\\  \end{smallmatrix} \right)\right)$. Since  $\delta^*$ normalizes $Q$, $W\Phi(S)$ and $R$.  Hence $\delta^*$ acts as a scalar on $S/\Phi(S)$.  We calculate $c_d$ scales $R/\Phi(S)$ by $\lambda$ and $Q/\Phi(S)$ by $t$. Hence $t = \lambda$. Now Equation~\ref{action} shows that $\delta^*$ scales $z$ by $t^2\lambda^3= \lambda^5$.  As $\delta$ scales $W/Z$ as $\delta^*$ scales $W\Phi(S)/\Phi(S)$ and $\delta$ scales $Z$ by the inverse of this (as $\delta$ as determinant $1$), we have $\lambda^5= \lambda^{-1}$.  Thus we have $\lambda^6=1$ and we conclude that $p=7$.

Let $D \le \Aut_\F(S)$ denote the subgroup generated by the extensions of the automorphisms in $N_{\Aut_\F(W)}(\Aut_S(W))$. Then
  $Q \ne W\Phi(S)\ne R$, are invariant under the action of $D$  and so $D$ acts as scalars on $S/\Phi(S)$. Therefore  $|D\Inn(S)/\Inn(S)|\le 6$. On the other hand, as $\delta^*\in D$,  $|N_{\Aut_\F(W)}(\Aut_S(W))/\Aut_S(W)|\ge 6$ with equality if and only if $\Aut_\F(W)= O^{p'}(\Aut_\F(W))$. This proves (a) and part (b) follows as $D=N_{\Aut_\F(S)}(W)$.

Let $\delta^* \in D$ have order $6$.  Then $\theta=\delta^*|_Q$ induces a faithful action on $\Out_S(Q)\cong S/Q$ and, as $\theta|_Z=\delta^*|_Z = (\delta^*|_W)|_Z= \delta|_Z$  acts faithfully on $Z$, we see that (c) holds.
\end{proof}

We can now prove Theorem \ref{t:esslist}.

\begin{proof}[Proof of Theorem \ref{t:esslist}]
This follows from Lemmas \ref{l:econtq}, \ref{l:2cases}, \ref{l:Nop3} and \ref{l:wine}.
\end{proof}
\section{Determining the fusion systems up to isomorphism when $p \ge 5$}\label{s:determining}
Our hypotheses for this section are that $p \ge 5$, $\F$ is a saturated fusion system on $S$, a  Sylow $p$-subgroup of $\Gee_2(p)$, with $O_p(\F)=1$ and $\E$ is the set of $\F$-essential subgroups of $S$. Here is the result we shall prove:

\begin{Thm}\label{t:main}
Suppose that  $p \geq 5$, $S$ is a Sylow $p$-subgroup of $\Gee_2(p)$ and $\F$ is a saturated fusion system on $S$ with $O_p(\F)=1$. Then either $\F$ is isomorphic to the fusion system of  $\Gee_2(p)$ on $S$ or else $p \le 7$ and $\F$ is isomorphic to a subsystem of $p'$-index in one of the fusion systems listed in Table \ref{table1}. Furthermore in each row of Table \ref{table1}, columns 3-6 determine (up to isomorphism) at most one saturated fusion system on $S$.
\begin{table}[!ht]
  \renewcommand\thetable{5.1}
\centering
\caption{Exceptional fusion systems $\F$ on $S$ with $O_p(\F)=1$}
\label{table1}
\begin{tabular}{|c|c|c|c|c|c|c|c|c|}
\hline
&$p$ & $\Out_\F(W) $& $\Out_\F(R)$&$\Out_\F(Q)$&$\Out_\F(S)$&\rm{Example}&$\Gamma_{p'}(\F)$\\
\hline
$\F_5^0$&$5$ &--&$\GL_2(5)$&$2\udot \Alt(6). 4$&$4 \times 4$&$\Ly$&$1$\\
$\F_5^1$&$5$ &--&$\GL_2(5)$&$4 \circ 2_{-}^{1+4}.\Frob(20)$&$4 \times 4$&$\Aut(\HN)$&$2$\\
$\F_5^2$&$5$ &--&$\GL_2(5)$&$2_{-}^{1+4}.\Alt(5).4$&$4 \times 4$&$\BM$&$1$\\
$\F_7^{0}$&$7$&--&$\GL_2(7)$&$3 \times 2\udot \Sym(7)$&$6\times 6$&--&$1$\\
$\F_7^1(1_1)$&$7$ & $\SL_2(7)$ &--&--&$6 $&--&$1$\\
$\F_7^1(2_{1})$&$7$ & $\SL_2(7)$ &--&--&$6 $&--&$1$\\
$\F_7^1(2_2)$&$7$ & $\SL_2(7)$ &--&--&$6 $&--&$1$\\
$\F_7^1(2_3)$&$7$ & $\SL_2(7)$ &--&--&$6 \times 2 $&--&$2$\\
$\F_7^1(3_1)$&$7$ & $\SL_2(7)$ &--&--&$6 $&--&$1$\\
$\F_7^1(3_2)$&$7$ & $\SL_2(7)$ &--&--&$6 $&--&$1$\\
$\F_7^1(3_3)$&$7$ & $\SL_2(7)$ &--&--&$6 $&--&$1$\\
$\F_7^1(3_4)$&$7$ & $\SL_2(7)$ &--&--&$6 \times 3$&--&$3$\\
$\F_7^1(4_1)$&$7$ & $\SL_2(7)$ &--&--&$6 $&--&$1$\\
$\F_7^1(4_2)$&$7$ & $\SL_2(7)$ &--&--&$6 $&--&$1$\\
$\F_7^1(4_3)$&$7$ & $\SL_2(7)$ &--&--&$6 \times 2 $&--&$2$\\
$\F_7^1(5)$&$7$ & $\SL_2(7)$ &--&--&$6 $&--&$1$\\
$\F_7^1(6)$&$7$ & $\SL_2(7)$ &--&--&$6 \times 6$&--&$6$\\
$\F_7^2(1)$&$7$ & $\SL_2(7)$ &$\SL_2(7).2$&--&$6\times 2$&--&$1$\\
$\F_7^2(2)$&$7$ & $\SL_2(7)$ &$\SL_2(7).2$&--&$6\times 2$&--&$1$\\
$\F_7^2(3)$&$7$ & $\SL_2(7)$ &$\GL_2(7)$&--&$6\times 6$&--&$3$\\
$\F_7^3$&$7$ &$\SL_2(7)$  &--&$\GL_2(7)$&$6\times 6$&--&$1$\\
$\F_7^4$&$7$ &$\SL_2(7)$  &--&$3 \times 2\udot \Sym(7)$&$6\times 6$&--&$1$\\
$\F_7^5$&$7$ &$\SL_2(7)$  &$\GL_2(7)$&$\GL_2(7)$&$6\times 6$&--&$1$\\
$\F_7^{6}$&$7$ &$\SL_2(7)$  &$\GL_2(7)$&$3 \times 2\udot \Sym(7)$&$6\times 6$&$\Mo$&$1$\\
\hline
\end{tabular}
\end{table}
\end{Thm}

A description of the fusion systems in Table~\ref{table1} is developed throughout this section. Especially for the fusion systems $\F_7^1(j_i)$ see the discussion surrounding Notation~\ref{not}. One further remark on the notation: the subscript indicates the prime $p$ while the superscript just assists in distinguishing the different systems.
Recall from Lemma~\ref{l:outfsid} that, since we may adjust $\F$ by an automorphism of $S$, we may assume
$$\Aut_\F(S) \text{ is a subgroup of } \Aut_B(S)$$
and so
$$\Out_\F(S) \text{ is a subgroup of } \Out_B(S).$$

We start by   presenting an important preliminary result  for the case when $Q \in \E$.

\begin{Lem}\label{l:possforAQ} Suppose $p \ge 5$, $Q \in \E$ and assume
\begin{enumerate}
\item there exists $\theta \in N_{\Aut_\F(Q)}(\Aut_S(Q))$  such that $\theta$   induces an automorphism of order $p-1$ on both $\Out_{S}(Q)$ and $Z(Q)$; and
\item  if $p=5$  then $\det \theta|_{Z_2}=1$.
\end{enumerate}
Then  $\Out_\F(Q)$ is $\Out(Q)$-conjugate to one of the subgroups in the following list:
\begin{enumerate}
\item[(a)] $p=5$ and  $\Out_\F(Q) \sim 2\udot \Alt(6).4$;
\item[(b)] $p=5$ and $\Out_\F(Q) \sim  4 \circ  2 ^{1+4}_-. \Frob(20)$;
\item[(c)] $p=5$ and $\Out_\F(Q) \sim 2 ^{1+4}_-. \Frob(20)$;
\item[(d)] $p=5$ and $\Out_\F(Q) \sim 2 ^{1+4}_-. \Alt(5).4$;
\item[(e)] $p=5$ and $\Out_\F(Q) \cong \GL_2(5)$  with $\Out_\F(Q)$ acting  reducibly on $Q/Z$ normalizing $Z_3/Z$;
\item[(f)] $p=7$ and $\Out_\F(Q) \sim 3 \times 2\udot \Sym(7)$; or
\item[(g)] $p\ge 5$ and $\Out_\F(Q) \cong \GL_2(p)$.
\end{enumerate}
Furthermore, either $\Out_\F(S)= \Out_B(S)$ is isomorphic to $(p-1)\times (p-1)$ or else $p=5$,  case (c) holds, $\Out_\F(S) \cong 4 \times 2$ has index $2$ in $\Out_B(S)$ and $\Out_\F(Q)$ is the subgroup listed in (b). Finally, if $\Out_\F(Q)$ is one of the groups listed in (a)-(g) then $N_\F(Q)$ is uniquely determined up to isomorphism and in particular $\Aut_\F(Q)$ and $\Aut_\F(S)$ are uniquely determined. \end{Lem}

\begin{proof} Recall that $\Out(Q) \cong \GSp_4(p)$ by  Proposition~\ref{l:autQ}. Thus, as $\Aut_\F(Q) \ge \Inn(Q)$,
we are required to find all the possibilities for $\Out_\F(Q)$ up to $\GSp_4(p)$-conjugacy. Set   $V= Q/Z$, $G= \Out_\F(Q)$  and $\Gamma= \GSp_4(p) = \GSp(V) \le \GL(V)$.  We know that $\Out_S(Q)= S/Q $ has order $p$ and, by  Lemma~\ref{l:charz} (f), the non-trivial elements of $\Out_S(Q)$ act on $Q/Z$ with a single Jordan block. Moreover, as $Q \in \mathcal E$, $\Out_S(Q)$ is not normal in $G$.

Suppose that  the projection  of $G$ into $\mathrm{PGL}(V) $ is almost simple.
Then  \cite[Theorem 4.1]{COS}    yields candidates for $O^{p'}(G)$: if $V$ is irreducible $p=5$ with $O^{5'}(G) \cong 2\udot \Alt(6)$, $p=7$ with $O^{7'}(G)\cong 2\udot \Alt(7)$ with $p=7$ or $p \ge 5$ is arbitrary and $O^{p'}(G)\cong \SL_2(p)$.
If $V$ is not irreducible, we have $p=5$ and  $O^{5'}(G) \cong \SL_2(5)$.
Assuming that $V$ is irreducible,
\cite[Tables 8.12 and 8.13]{BHR} shows that all the candidates for $O^{p'}(G)$ in (b) exist and are unique up to conjugacy  in $\Gamma$.  Furthermore, as $O^{p'}(G) \le \Sp_4(p)$, we obtain $\langle \theta\rangle \cap O^{p'}(G) =1$, $G=N_\Gamma(O^{p'}(G))= O^{p'}(G)\langle \theta \rangle$ and the information provided in \cite[Tables 8.12 and 8.13]{BHR} (and Schur's Lemma) gives the details listed in (a), (f) and (g).

In the case $V$ is indecomposable, the $2$-space preserved by $G$ is isotropic. Thus $G$ is contained in a maximal parabolic subgroup $P$ of $\Gamma$ which leaves an isotropic $2$-space invariant.  To see uniqueness here, we note that the 1-cohomology of the 3-dimensional $\mathbb F_5\SL_2(5)$-module has dimension 1 (see \cite[Lemma 3.11]{COS}). Thus there are five  $O^{5'}(P)$ conjugacy classes  of subgroups isomorphic to $\SL_2(5)$ contained in  $O^{5'}(P)$.   One of these acts completely reducibly on $V$ and the others are all conjugate by an element of order $4$ in $P$.  Thus $O^{5'}(G)$ is uniquely determined and since $\theta \in G$ induces an element of order $4$ on $Z$, we have $G\cong \GL_2(5)$. This is case (e).

  Suppose that the projection  of $G$ into $\mathrm{PGL}(V) $ is not an almost simple group. Then, by \cite[Proposition 4.4]{COS}, $p=5$, $P=O_{5'}(G)= F^*(G)$ is isomorphic to one of $\mathrm 4 \circ 2^{1+4}$ or $2^{1+4}_{-}$   and either
\begin{itemize}
\item $G/P = \Sym(6);$
\item $G/P= \Sym(5);$ or
\item $G/P= \Frob(20)$.
\end{itemize}
By \cite[Tables 8.12 and 8.13]{BHR}, the first case  cannot occur.

 In $\Gamma=\GSp_4(5)$,  $P$ is uniquely determined up to $\Gamma$-conjugacy. It follows that $H=N_{\Gamma}(P)$ is also uniquely determined  up to conjugacy in $\Gamma$.

Suppose $G/P\cong \Sym(5)$.
Then, as $G/C_G(Z(Q))$ is cyclic of order $4$ generated by the image of $\theta$, we see that $O^{5'}(G) \sim 2^{1+4}_-.\Alt(5)$ and $G=  N_{\Gamma}(P)$. This is the configuration in (d) and it contains $\langle Z(\Gamma), \theta\rangle$ of order 16.

Suppose $G/P\cong \Frob(20)$. Then $$G \le N_{\Gamma}([P,\Out_S(Q)]\Out_S(Q))\sim  4 \circ 2^{1+4}. \Frob(20).$$  It follows that if $P \cong  4 \circ 2^{1+4}$, then $G$ is uniquely determined and again it contains $\langle Z(\Gamma), \theta\rangle$ of order 16. This is listed as (b).  If $P \cong 2^{1+4}_-$, then $N_{\Gamma}(P\Out_S(Q))/P\Out_S(Q)$ is abelian of type $2 \times 4$. It follows that $N_{\Gamma}(P\Out_S(Q))$ contains exactly two candidates for $G$. However, $\theta \in G$ and so we know in this case that $$G= P\Out_S(Q)\langle \theta\rangle.$$ To see that this group is unique, we    show that $\langle \theta\rangle$ is uniquely determined as a subgroup of $\langle Z(\Gamma),\theta\rangle$ and this is where we   hypothesis (2). In the   case that $3$ does not divide $p-1$, we have that $\langle Z(\Gamma),\theta\rangle$ acts faithfully on $Z_2$ because the elements of $Z(\Gamma)$ scale $V$ by some  $\omega\in \mathbb F_5$ and then $Z$ by $\omega^2$ (so the determinant $1$ elements in $Z(\Gamma)$ have order dividing $3$). This means that when $p=5$, $\langle \theta\rangle$ is uniquely determined as the subgroup of $\langle Z(\Gamma),\theta\rangle$ consisting of those elements which have determinant $1$ on $Z_2$. This gives (c).

Now we observe that in all cases other than (c), $N_{\Out_\F(Q)}(\Out_S(Q)) \cong p:(p-1)^2$. In case (c),  we have already remarked that $ N_{\Out_\F(Q)}(\Out_S(Q)) \cong 5:(4 \times 2).$
By saturation these morphisms   lift to elements of $\Aut_\F(S)$ and so we have the order and isomorphism type of $\Out_\F(S)$.

It remains to prove the final uniqueness statement. Assume that $\Out_\F(Q)$ is one of the subgroups listed in (a)-(g). Then since $\Out_\F(Q)$ is uniquely determined up to conjugacy in $\Out(Q)$, $\Aut_\F(Q)$ is uniquely determined up to conjugacy in $\Aut(Q)$. Since $\F$ is a saturated fusion system \cite[Theorem I.4.9]{AKO} uniquely determines a group $M$ with Sylow $p$-subgroup $S$ such that $N_\F(Q)=\F_S(M)$. Hence $\Aut_\F(S)=\Aut_{N_\F(Q)}(S)=\Aut_M(S)$ is uniquely determined.
\end{proof}
%Assume that $H$ and $K$ are distinct candidates for $\Out_\F(Q)$ in $\Out(Q)$ with $H \cong K$. Then $H \cap K \ge   N_{\Out_\F(Q)}(\Out_S(Q))$ and $H$ and $K$ are conjugate in $\Out(Q)$. Hence there exists $\mu \in \Out(Q)$ such that $H^\mu = K$.  As $\Out_S(Q) \in \Syl_p(H) \cap \Syl_p(K)$, we may suppose that $\mu \in N_{\Out(Q)}( N_{\Out_\F(Q)}(\Out_S(Q)))$. Now notice that $N_{\Out_\F(Q)}(\Out_S(Q))$ has order $p(p-1)^2$ or $p=5$ and $|N_{\Out_\F(Q)}(\Out_S(Q))|=2^35$. Thus, by Lemma~\ref{l:matrixcalc} (c), either $N_{\Out(Q)}( N_{\Out_\F(Q)}(\Out_S(Q)))= N_{\Out_\F(Q)}(\Out_S(Q)) \le H$ or $p=5$, case (c) holds and  $N_{\Out(Q)}( N_{\Out_\F(Q)}(\Out_S(Q)))\Out_\F(Q)$ is the example in (b). Either way, we have $K=H^\mu = H$.  Hence $\Out_\F(Q)$ is uniquely determined by its isomorphism type as a subgroup of $\Out(Q)$ containing $N_{\Out_\F(Q)}(\Out_S(Q))$.

The next lemma unlocks the results from Lemma~\ref{l:possforAQ}.

\begin{Lem}\label{l:AFQ1} Suppose that $R\in \E$. Then
\begin{enumerate}
\item [(a)] there exists $\theta \in N_{\Aut_\F(Q)}(\Aut_S(Q))$  such that $\theta$   induces an automorphism of order $p-1$ on both $\Out_{S}(Q)$ and $Z$; and
\item [(b)]  if $p=5$  then $\det \theta|_{Z_2}=1$.
        \end{enumerate}
\end{Lem}

\begin{proof} By Lemma~\ref{l:AFR1},  $O^{p'}(\Out_\F(R))=\SL_2(p)$ and this group acts faithfully on $Z_2=Z(R)$ and on $R/\Phi(R)= R/Z_3$. Thus there exists $\theta_0\in N_{\Aut_\F(R)}(\Aut_S(R))$ such that $\theta_0$ induces an automorphism of $Z$ of order $p-1$ and has determinant $1$ when acting on $Z_2$. Since $\F$ is saturated, $\theta_0$ extends to an element of $\Aut_\F(S)$ and then by restriction we obtain an element $\theta$ of $\Aut_\F(Q)$ which acts on $Z$ with order $p-1 $. Now we note that $\theta_0$ acts on $R/Z_4=R/(Q\cap R)\cong RQ/Q= S/Q$ faithfully and so we also have   $\theta$   induces an automorphism of order $p-1$ on   $\Out_{S}(Q)$ .
\end{proof}

\subsection{The case $\mathcal E \subseteq \{ Q,R\}$}

By Theorem~\ref{t:esslist}, $\mathcal E \cap \mathcal W \ne \emptyset $ implies that $p= 7$ and so the typical case occurs when $\mathcal E \subseteq \{ Q,R\}$.   We consider this scenario in this subsection.

\begin{Lem}\label{l53} If $\mathcal E \subseteq \{ Q,R\}$, then $\mathcal E= \{Q,R\}$.
\end{Lem}
\begin{proof} Suppose that $\mathcal E  $ has a unique element $X \in \{Q,R\}$. Then $$\mathcal F = \langle \Aut_\F(X), \Aut_\F(S)\rangle$$ by Theorem~\ref{t:alp}. By Lemma~\ref{l:charz} (d), $X$ is a characteristic subgroup of $S$ and thus we see that $1\ne X= O_p(\F)=1$, which is a contradiction.
\end{proof}

\begin{Lem}\label{l:exccase}
If $\E \subseteq \{Q,R\}$, then either \begin{enumerate} \item [(a)] $\Out_\F(S)=\Out_B(S)$ has order $(p-1)^2$ and $\Aut_\F(R) \cong \GL_2(p)$; or
\item [(b)]$p=5$, Lemma \ref{l:possforAQ} case (c) holds, $\Out_\F(S)$ has order $8$ and $\Out_\F(R) \cong 4 \circ \SL_2(5)$.
\end{enumerate} In both  cases $|C_{\Out_\F(R)}(\Out_S(R))| > 2$.
\end{Lem}

\begin{proof}  Lemmas~\ref{l:possforAQ}, \ref{l:AFQ1} and \ref{l53}  combine to give the possibilities for $\Out_\F(S) $. We calculate that for $d \in B$ of the form
  As the elements of $\Aut_\F(S)$ restrict  to members of $N_{\Aut_\F(R)}(\Aut_S(R))$ and $\Out_S(R)$ has order $p$, if $|\Out_\F(S)|$ has order $(p-1)^2$ then  $|C_{\Out_\F(R)}(\Out_S(R))| = p-1>2$. So suppose that Lemma \ref{l:possforAQ} case (c) occurs. Recall that $O^{p'}(\Out_\F(R))  \cong \SL_2(5)$. Hence, by a Frattini argument $$\Out_\F(R)=N_{\Out_\F(R)}(\Out_S(R))O^{p'}(\Out_\F(R))  \sim \SL_2(5).2$$ and this group is a subgroup of $\GL_2(5)$. From these observations we conclude that   $\Out_\F(R) \cong 4 \circ \SL_2(5)$. This proves the last part of  the claim.
\end{proof}

 We next show that $\Aut_\F(S)$ uniquely picks out a subgroup of $\Aut(R)$ to play the role of $\Aut_\F(R)$.

\begin{Lem}\label{l:afrunique}
 If $\E \subseteq \{Q,R\}$, then $\Aut_\F(R)$ is uniquely determined as a subgroup of $\Aut(R)$.
\end{Lem}

\begin{proof}
 Since $R \in \E$, $\Out_\F(R)$ is isomorphic to a subgroup of $\GL_2(p)$ and $O^{p'}(\Out_\F(R)) \cong \SL_2(p)$ by Lemma \ref{l:AFR1}. By Lemma \ref{l:exccase} we have $\Out_\F(R)$ is isomorphic to $\GL_2(p)$ or Lemma \ref{l:possforAQ} case (c) holds with $\Out_\F(R)\cong 4 \circ \SL_2(5)$. Set $$\bar{Y}=N_{\Out_\F(R)}(\Out_S(R)) \mbox{ and } \bar{T}=\Out_S(R).$$ Then $C_{\bar{Y}}(\bar{T})$ has order greater than 2 by Lemma \ref{l:exccase}. Thus Lemma \ref{l:runique} implies that $\Aut_\F(R)$ is uniquely determined as a subgroup of $\Aut(R)$.
\end{proof}

\begin{Lem}\label{l:unique}
 If $\E \subseteq \{Q,R\}$, then $\E = \{Q,R\}$ and $\F$ is uniquely determined   by specifying
    the  subgroup of $\Aut(Q)$ from Lemma~\ref{l:possforAQ} which is $\Aut_\F(Q)$.
\end{Lem}

\begin{proof}  For this we just  coalesce Lemmas~\ref{l:possforAQ}, \ref{l:exccase}(a) and  \ref{l:afrunique}.
\end{proof}

\begin{Lem}
Suppose that $\E \subseteq \{Q,R\}$. Then $\Gamma_{p'}(\F)=1$ or $p=5$,   Lemma~\ref{l:possforAQ} (b) holds and $\Gamma_{5'}(\F)=2$.
\end{Lem}

\begin{proof} We have $O^{p'}(\Out_\F(R)) \cong \SL_2(p)$ and so this group contributes a cyclic group of order $p-1$ which acts faithfully on $Z$ to $\Out_\F^0(S)=\Aut_\F^0(S)/\Inn(S)$. If $N_{O^{p'}(\Out_\F(Q))}(\Out_S(Q))$ has order $p-1$, then, as $O^{p'}(\Aut_\F(Q))$ centralizes $Z$, we have $|\Out_\F^0(S)|=(p-1)^2$ and we obtain $\Gamma_{p'}(\F)=1$. Now examining the groups in listed in Lemma~\ref{l:possforAQ}, yields that the  only possibility for   $\Gamma_{p'}(\F)$ to be non-trivial arises when $p=5$ and Lemma~\ref{l:possforAQ} (b) holds. In this case $\Out_\F(S) \cong 4 \times 4$ and $\Out_\F^0(S)$ has index $2$.
\end{proof}

\begin{Thm}\label{t:case esubsetqr}
Suppose that $\E \subseteq \{Q,R\}$ and $\F$ is a saturated fusion system on $S$ with $O_p(\F)=1$.  Then $\E= \{Q,R\}$ and $\F$ is isomorphic to the fusion system of $\Gee_2(p)$ or to $\F_5^0$,   $\F_5^1$, $\F_5^2$ or $\F_7^0$  or to a subsystem of index $2$ in $\F_5^1$ as listed Table \ref{table1}.
\end{Thm}

\begin{proof}   From Lemma~\ref{l:unique}, $\F$ is uniquely determined once $\Out_\F(Q)$ is specified. Thus we only need to check that $O_p(\F)=1$.  If $\Aut_\F(Q)$ acts irreducibly on $Q/Z$, then the only candidates for $O_p(\F)$ are $Z$ and $Q$. Since $O_p(\F)$ is contained in all the $\F$-essential subgroups and  $\Aut_\F(R)$ does not normalize $Z$, we are done.  The only possibility which arises with $\Aut_\F(Q)$ acting  reducibly on $Q/Z$, occurs in Lemma~\ref{l:possforAQ} (e).  In this case, $p=5$ and using the detail in Lemma~\ref{l:possforAQ} (e), we see that $\Aut_\F(Q)$ leaves invariant the unique normal subgroup of $S$ of order $5^3$. That is $\Aut_\F(Q)$ leaves $Z_3$ invariant. Since $Z_3=\Phi(R)$ is also invariant under the action of $\Aut_\F(R)$, we have $O_5(\F)= Z_3$ in this case, a contradiction.
\end{proof}

\subsection{The case  $\mathcal E \cap \mathcal W \ne \emptyset$}
In this subsection, we assume that $\mathcal E \cap \mathcal W \ne \emptyset$ and consequently $p=7$ by Lemma~\ref{l:wine}. Since $\Gee_2(7)$ has a $7$-dimensional representation over $\mathbb F_7$, $S$ has exponent $7$. In fact, as $S$ is now a fixed group, we may use {\sc Magma} \cite{magma} to perform calculations in $S$ and also  to calculate in the automorphism group of   subgroups of $S$.

Motivated by Lemma~\ref{l:wine}, for an arbitrary subgroup  $W \in \W \cap \E$ we define  $$\Delta= N_{\Aut_\F(S)}(W)\Inn(S) =\left\langle c_d, \Inn(S)\mid  d=\left(\lambda,\left(\begin{smallmatrix} \lambda &0\\0&1\end{smallmatrix}\right)\right)\in B, \lambda \in \mathbb F_7^\times\right\rangle.$$
 Thus  $\Delta/\Inn(S)$ is cyclic of order $6$.

\begin{Lem}\label{l:36outfs}
The following hold:
\begin{enumerate}
 \item[(a)] $\Delta$ has six orbits on $\W$;
 \item [(b)] for $W_1, W_2 \in \W$,  $W_1\Phi(S)= W_2\Phi(S)$ if and only if $W_1$ and $W_2$ are in the same $\Delta$-orbit;
\item  [(c)]  $\Aut(S)$ acts transitively on  $\W$;
\item [(d)] $\W$ is the union of $\frac{36}{|\Out_\F(S)|}$ $\F$-conjugacy classes of subgroups of order 49.
\end{enumerate}
\end{Lem}

\begin{proof}
As $S$ has exponent $7$, the number of subgroups of $S$ of order $49$ which are not contained in $Q$ or $R$ and contain $Z$  is $$|\W|= \frac{|S|-|Q|-(|R|-|R \cap Q|)}{(49-7)} = \frac{7^6-7^5-(7^5-7^4)}{42}=7^3.6.$$ Here we use the fact that $W_x=W_{x'}$ if and only if $x' \in W_{x} \backslash Z$, where $x\in S \backslash {Q \cup R}$ and $W_x$ is as defined in Notation \ref{n:wx}. Let $W\in \W$.  Then, as $W \not \le Q$ and $W \not\le R$, $N_S(W)= WZ_2$. Thus $|W^S|= |S:WZ_2|= 7^3$ and so $\W$ is the union of six $S$-orbits and also six $\Delta$-orbits. This proves (a).  Now   $W_x, W_y \in \W$ are in the same $\Delta$-orbit if and only if $W_x\Phi(S)= W_y \Phi(S)$ which is (b). Since $\Aut(S)/C_{\Aut(S)}(S/\Phi(S))=\Aut_B(S)C_{\Aut(S)}(S/\Phi(S))/C_{\Aut(S)}(S/\Phi(S))$ acts as diagonal matrices on $S/\Phi(S)$ by Lemma~\ref{l:liftnorm}, we see that $\Aut(S)$ acts transitively on $\mathcal X = \{W\Phi(S)\mid W \in \W\}$ and hence also on $\W$. Thus (c) holds.

Now, for $W\in \W \cap \E$,  $|N_{\Aut_\F(S)}(W)\Inn(S)/\Inn(S)| = 6$ by Lemma~\ref{l:wine} (b).  Therefore $\Out_\F(S)$  has  $36/|\Out_\F(S)|$ orbits on $\mathcal X$ and so there are $36/|\Out_\F(S)|$  $\F$-conjugacy classes. This proves (d).
\end{proof}

\begin{Lem}\label{l:possAQ2} If $Q \in \E$, then \begin{enumerate} \item[(a)]  $\Out_\F(S) =\Out_B(S)\cong 6\times 6$;
\item [(b)]  $\Aut_\F(Q)$ is a uniquely determined subgroup $\Aut(Q)$; and\item[(c)] either $\Out_\F(Q)\cong \GL_2(7)$ or $\Out_\F(Q)\cong 3 \times 2 \udot \Sym(7)$.\end{enumerate} Furthermore, $\Aut_\F(S)$ acts transitively on $\W$. \end{Lem}
\begin{proof}  Combining Lemmas~\ref{l:wine}(c) and \ref{l:possforAQ}  gives parts (a), (b) and (c).  Lemma~\ref{l:36outfs} shows that $\Aut_\F(S)$ acts transitively on $\W$.\end{proof}

\begin{Lem}\label{l:WandR}
If $\{Q,R\} \subset \E$, then $\Aut_\F(R)$ is uniquely determined, $\Out_\F(R) \cong \GL_2(7)$ and $\Out_\F(S) \cong 6\times 6$.
\end{Lem}

\begin{proof} The proof of the uniqueness of $\Aut_\F(R)$ follows the same steps  as in Lemma~\ref{l:afrunique}.
\end{proof}

\begin{Thm}\label{t:7andQ}
If $Q \in \E$ and $\E \cap \W \ne \emptyset$, then $\F$ is isomorphic to either $\F_7^3$, $\F_7^4$, $\F_7^5$ or $\F_7^6$.
\end{Thm}

\begin{proof} This follows by collecting the results of  Lemmas~\ref{l:wine}, \ref{l:possAQ2} and \ref{l:WandR}.
\end{proof}

We now move on to the case where $\E \subseteq \W \cup\{R\}$.

\begin{Not}\label{not}
Suppose that $I=\mathbb F_7^\times= \{1,2,3,4,5,6\}$.  Then the action of $\mathbb F_7^\times$ by multiplication on non-empty subsets of $I$  has orbit representatives  as follows.

\begin{eqnarray*}&& 1_1= \{1\},\\&& 2_1= \{1,2\}, 2_2=\{1,3\}, 2_3= \{1,6\},\\&&3_1 =\{1,2,3\}, 3_2=\{1,2,5\},3_3=\{1,2,6\}, 3_4= \{1,2,4\} ,\\&&
4_1=\{1,2,3,4\}, 4_2=\{1,2,3,5\}, 4_3=\{1,2,5,6\},\\
&&5_1=\{1,2,3,4,5\},\\&&6_1=\{1,2,3,4,5,6\}. \end{eqnarray*}
Observe that these orbits are regular other than  $2_3$, $4_3$ (both of which have length $3$), $3_4$ (which has length $2$ and $6_1$ which has length $1$.)
\end{Not}

By Lemma \ref{l:36outfs} $$\X= \{W\Phi(S)/\Phi(S)\mid W\in \W\}$$  consists of the six diagonal subgroups to $Q /\Phi(S)$ and $R/\Phi(S)$ in $S/\Phi(S)$ and the action of $\Aut_B(S)$ on $\X$ can be identified with the action of $\mathbb F_7^\times $ on $I$. In particular, $\Delta$ is contained in the kernel of this action.
This means that, if  $\mathcal Y$ is a union of $\Delta$-orbits on $\W\cap \mathcal E$, then $\{W\Phi(S)/\Phi(S)\mid W \in \mathcal Y\} \subseteq \X$. Since the elements of $\X$ correspond to $\Delta$-orbits on $\W$, we may sensibly denote the $\Delta$-orbits on $\W$ by $\W_i$ where $i\in I$. Now the $\Aut_B(S)$-orbits on the non-empty subsets of the set of $\Delta$-orbits $\{\W_1, \dots, \W_6\}$ on $\W$  have representatives as described in Notation~\ref{not}.  We may suppose that there exists $W_1\in \mathcal \W \cap \E$ such that $W_1\in \W_1$.  Of course $\W \cap \E$ is a union of $\Delta$-orbits and so corresponds to a subset $j$ of $I$ and any $\Aut_B(S)$ translate of $j$ corresponds to an isomorphic fusion system.  Thus we may suppose that $\W \cap \E$ corresponds to one of the subsets listed in Notation~\ref{not}. Now given fusion systems   $\F_1$ and $\F_2$ on $S$  with $\Aut_{\F_i}(S) \le \Aut_B(S)$ and $\W \cap E \ne \emptyset$, for $\F_1$ and $\F_2$ to be isomorphic, the corresponding subsets of $I$ must be $\Aut_B(S)$-conjugate. Thus, if $\W \supseteq \E$, to uniquely specify a fusion system, we need to specify a subset $j$ of $I$ to correspond to the $\Delta$-orbits on $\E$ and then a  subgroup of $\Aut_B(S)$ containing $\Delta$ and stabilizing $j$.

Let $j_i$ be a subset of $I$ as in Notation~\ref{not} and define $$B(j_i)= \Stab_{\Aut_B(S)}(j_i).$$
For an orbit representative $j_i$, define the fusion systems $$\G({j_i}) = \langle \Aut_{\F}(W), \Delta\mid W \in \W_k, k \in j_i\rangle$$
and then put $$\F_7^1(j_i) = \langle \G({j_i}), B(j_i)\rangle.$$

\begin{Thm}\label{t: case esubsetw}
Suppose that  $\E  \subseteq \W$.  Then   $\F$ is isomorphic to a subsystem of $7'$-index of $\F_7^1(j_i)$ containing $\G(j_i)$ where $j_i$ is an $\Aut_B(S)$-orbit on the non-empty subsets of  $\{\W_1, \dots, \W_6\}$. Furthermore, if these fusion systems are saturated then no two of them are isomorphic.
\end{Thm}

\begin{proof} The claim follows from the previous discussion.
\end{proof}

We remark that the set of $\F_7^1(j_i)$-essential subgroups in $\W$ is exactly $\bigcup\limits_{k \in j_i} \W_k$.

\begin{Thm}\label{case r in e} Suppose $R \in \E$ and $Q \not \in \E$. Then
$\Aut_\F(R) $ is uniquely determined and either
\begin{enumerate}
\item $\Out_\F(R) \cong \GL_2(7)$, $\Out_\F(S)=\Out_B(S) \cong 6 \times 6$ and $\Aut_\F(S)$ acts transitively on $\W$; or
\item $\Out_\F(R) \sim  \SL_2(7).2$, $\Out_\F(S) \cong 6 \times 2$ is uniquely determined in $\Out_B(S)$ containing $\Delta$ and $\Aut_\F(S)$ has three orbits each of length two on  $\W$.
\end{enumerate}
In particular, $\F$ is  isomorphic to either $\F_7^2(1)$, $\F_7^2(2)$, $\F_7^2(3)$ or $O^{7'}(\F_7^2(3))$.
\end{Thm}

\begin{proof} Let $W \in \W \cap \E$ and $\widetilde \Delta$ represent the subgroup of $\Aut_\F(R)$ obtained by restricting the morphisms in  $\Delta$ to $R$.
By Lemma~\ref{l:wine}, $\widetilde \Delta$ is generated by $\Aut_S(R)$ together with restrictions to $R$ of the elements $c_d$ where
$$d= \left({\lambda}, \left( \begin{smallmatrix}
\lambda & 0    \\
0 & 1 \\  \end{smallmatrix} \right)\right)$$ with $\lambda \in \mathbb{F}_7^\times$.
We calculate that  $\widetilde \Delta$ is cyclic of order $6$ and that on $R/\Phi(R)$ we can select a basis so that such elements act as diagonal matrices $\mathrm{diag} (\lambda ^2,\lambda)$  and so have determinant $\lambda^3$  which is a cube. Recall from Lemma~\ref{l:wine} (b) that $\widetilde \Delta$ is independent of the choice of $W\in \W$.   Thus $\Out_\F(R) \ge \langle O^{7'}(\Out_\F(R)), \widetilde \Delta\rangle \cong \SL_2(7).2$, the unique subgroup of $\GL_2(7)$ of index $3$. In addition, as $\widetilde \Delta$ acts as scalars on $S/\Phi(S)$, $\Out_S(R)$ admits $\widetilde \Delta$ faithfully. Now calculating in $\Aut(R)$ using  \textsc{Magma} \cite{magma} for example, we see that there is a unique subgroup  $X$ of $\Aut(R)$ containing $\Inn(R) $ with $X/\Inn(R) \cong \SL_2(7)$ which is normalized by   $\Aut_S(R)\widetilde \Delta$. Furthermore, $N_{\Aut(R)}(X)/\Inn(R) \cong \GL_2(7)$.  This means that $\Out_\F(R) \cong \SL_2(7).2$ or $\Out_\F(R)\cong \GL_2(7)$ and $\Aut_\F(R)$ is uniquely determined as a subgroup of $\Aut(R)$. In the respective cases we have $N_{\Aut_\F(R)}(\Aut_S(R))/\Aut_S(R) \cong 6 \times 2$ or $6\times 6$. The extension  of the morphisms in  this subgroup to $\Aut_\F(S)$   determine $\Aut_\F(S)$ to be  either the unique subgroup of index $3$ in $\Aut_B(S)$ containing $\Delta$ or $\Aut_B(S)$.  In particular, either  $\Aut_\F(S)$ has three orbits of length $2$ on $\X$ with representative of the first orbit being given by $2_3$ as in Notation~\ref{not} or $\Aut_\F(S)$ operates transitively on $\mathcal W$.

Hence,  if $\Out_\F(R)\cong \GL_2(7)$, then $\Aut_\F(S)= \Aut_B(S)$ is transitive on $\W$ and we have no choices to make. Thus in this case  $$\F= \langle \Aut_\F(R), \Aut_\F(W), \Aut_\F(S)\rangle$$ and this is the fusion system $\F_7^2(3)$.
Suppose that $\Aut_\F(S)$ has index $3$ in $\Aut_B(S)$. In this case, $\Aut_\F(R) \cong \SL_2(7).2.$ and, setting $\W_{k,\ell}= W_k\cup \W_\ell$,  the $\Aut_\F(S)$ orbits on $\W$ are $\W_{1,6}$, $\W_{3,4}$ and $\W_{2,5}$.  Hence, up to altering $\F$ by an element of $\Aut_B(S)$, we may suppose that one of the following holds
 \begin{eqnarray*}
 \F&=& \langle \Aut_\F(R), \Aut_\F(S), \Aut_\F(W)\mid W \in \W_{1,6}\rangle;\\
 \F&=& \langle \Aut_\F(R), \Aut_\F(S), \Aut_\F(W)\mid W \in \W_{1,6}\cup \W_{3,4}\rangle; \text{ or }\\
 \F&=&\langle \Aut_\F(R), \Aut_\F(S), \Aut_\F(W)\mid W \in \W\rangle.\\
  \end{eqnarray*}
  These fusion systems are $\F_7^2(1)$, $\F_7^2(2)$ and $O^{7'}(\F_7^2(3))$ respectively.
\end{proof}
We can now complete the proof of Theorem \ref{t:main}:

\begin{proof}[Proof of Theorem \ref{t:main}]
This follows from Theorem \ref{t:case esubsetqr} in the case $\W \cap \E = \emptyset$ and from Theorems \ref{t:7andQ}, \ref{case r in e} and \ref{t: case esubsetw} in the case $\W \cap \E \neq \emptyset$.
\end{proof}

\section{Completing the proof of Theorem \ref{t:main1} for $p\ge 5$}\label{s:final}

In order to complete the proof of Theorem \ref{t:main1}, it remains to show that each of the fusion systems described in Table \ref{table1} exists and is saturated, and to establish which ones are realizable as fusion systems of finite groups.

\begin{Thm}\label{They are saturated}
Each of the fusion systems listed in Theorem~\ref{t:main} exists and is saturated.
\end{Thm}

\begin{proof} Examining the list of maximal subgroups of $\Gee_2(p),$ $\Ly$, $\Aut(\HN)$, $\BM$  and $\M$ yields  that the fusion systems $\F_{\Gee_2(p)}(S),$ $\F_5^0,$ $\F_5^1$, $\F_5^2$ and $\F_7^6$ are respectively isomorphic to the fusion systems of $\Gee_2(p),$ $\Ly$, $\Aut(\HN)$,  $\BM$ and $\M$ on their Sylow $p$-subgroups. In particular, each of these fusion systems is saturated. Let $\F=\F_7^0$. Then $\E= \{Q,R\} $ and $N_\F(Q)$ and $N_\F(R)$ are saturated fusion systems on $S$ with $O_7(N_\F(Q))=Q$ and $O_7(N_\F(R))=R$. Hence by \cite[Theorem I.4.9]{AKO} there exist  finite groups $G_1$ and $G_2$ which realize $N_\F(Q)$ and $N_\F(R)$ respectively.
Moreover $O_p(G_1)=Q$ and  $O_p(G_2)=R$ and $Q$ and $R$ are self-centralizing in these groups.
 In addition, we may realize $N_\F(S)$ by $G_{12}$ which may be embedded into both $G_1$ and $G_2$. Note that this configuration  appears in the Monster sporadic simple group and so exists.
  Let $G^*$ be the free amalgamated product  $G_1*_{G_{12}}G_2$ and let $\Gamma$  be the coset graph $\Gamma(G^*, G_1, G_2, G_{12})$. Since the only $\F_{S}(G_1)$-essential subgroup is $Q$ and the only $\F_{S}(G_2)$-essential subgroup is $R$, to invoke  Theorem \ref{t:satamalg}, we only have to demonstrate that for any $\F_S(S)$-centric subgroup $A$, the fixed vertex set $\Gamma^A$ is finite.

For adjacent vertices $\alpha, \beta \in \Gamma$ with $\alpha$ a coset of $G_1$ and $\beta$ a coset of $G_2$, we set $Q_\alpha= O_7(G_\alpha)$ and $R_\beta = O_7(G_\beta)$. Thus $Q_\alpha$ is $G^*$-conjugate to $Q$ and $R_\beta$ is $G^*$-conjugate to $R$. We also set $S_{\alpha \beta}= O_7(G_{\alpha \beta})$ where $G_{\alpha\beta}= G_\alpha \cap G_\beta$.  Notice that $G_{\alpha \beta}= N_{G_\alpha}(S_{\a\b})= N_{G_\b}(S_{\a\b}) $ and $G_{\a\b}$ is a maximal subgroup of $G_\a$ and $G_\b$.

Assume that $A \le S_{\alpha \beta}$ is $S_{\alpha \beta}$-centric.  Seeking a contradiction we further assume that  $\Gamma^A$ is infinite.  Notice first that any $7$-group which stabilizes an arc $ \gamma, \delta, \epsilon$ of length $2$ is contained in $S_{\gamma\delta}\cap S_{\delta\epsilon }= O_7(G_\delta) $ which is one of $R_\delta $ or $Q_\delta$. Since $\Gamma^A$ is infinite, we may consider a path emanating from the arc $\alpha, \beta$ of infinite length. We choose notation so that $G_\alpha = G_2$ and $G_\beta= G_1$  and consider  a path $$\alpha, \beta, \gamma, \delta, \epsilon, \zeta, \eta$$ which is fixed by $A$.
Since $A$ stabilizes the arc $\alpha, \beta, \gamma, \delta$, $A$ is contained in $Q_\b\cap R_\gamma$.  In particular, $A \le Q_\beta$ and, as $A$ is $S_{\alpha\beta}$-centric,  $Z(Q_\b)= Z(S_{\alpha \beta}) \le A$. Thus $$Q_\beta\text{ normalizes  }A.$$
Notice that $Z(Q_\delta) \le Z(R_\gamma) \le Q_\beta \le S_{\alpha\beta}$. Therefore, using the fact that $A$ fixes the arc $\alpha, \beta, \gamma, \delta, \epsilon$, we deduce first that $A \le Q_\delta$ and second that $Z(Q_\delta) \le A$.  Now we have $$Z(R_\gamma)= Z(Q_\beta)Z(Q_\gamma) \le A\le Q_\b \cap R_\gamma \cap Q_\delta= \Phi(R_\gamma).$$ Since $\Phi(R_\gamma)$ is abelian and $\Phi(R_\gamma) \le Q_\beta \le S_{\alpha \beta}$, we now see that $$A= \Phi(R_\gamma)$$ because $A$ is $S_{\alpha \b}$-centric. In particular, $A$ is normalized by $G_\gamma$. Since $A$ fixes the arc $ \gamma, \delta, \epsilon, \zeta, \eta$ we have $$A \le Q_\delta \cap R_\epsilon \cap Q_\zeta=\Phi(R_\epsilon).$$ Since $|A| = 7^3$, we must have $A = \Phi(R_\epsilon)$. Hence $\Phi(R_\epsilon)= \Phi(R_\gamma)$ and this subgroup is normalized by $\langle G_{\gamma\delta}, G_{\delta \epsilon}\rangle = G_\delta$.
But then $A=\Phi(R_\gamma)$ is normalized by $\langle G_\gamma,G_\delta\rangle= G^*$ which is absurd. We conclude that $\Gamma^A$ is finite and thus that  $\F_7^0$ is saturated.

We are left only with the cases where $\E \cap \W \neq \emptyset$ and $\F$ is isomorphic to $\F_7^1(j_i)$, $\F_7^2(j)$ or  $\F_7^i$ with $3 \le i \le 5$.  Let $\E_0=\E \backslash (\E \cap \W)$ (so $\E_0 \subseteq \{Q,R\}$) and define $\F_0=\langle \Aut_\F(P) \mid P\in \E_0 \rangle$. Then $\F_0$ is saturated because in each case it is the fusion system of a finite group. We intend to apply Lemma \ref{t:proofsat} with $\{W_1,W_2, \ldots, W_m\}$ a set of representatives for the set of $\F_0$-conjugacy classes of subgroups in $\E \cap \W$. Observe that:
\begin{itemize}
\item[-] $W_i$ is $\F_0$-centric and minimal under inclusion amongst all $\F$-centric subgroups (note that any $\F$-centric subgroup must properly contain $Z(S)$);
\item [-] no proper subgroup of $W_i$ is $\F_0$-essential.
\end{itemize}
Hence the hypotheses of Theorem \ref{t:proofsat} are satisfied and $\F$ is saturated and exists as the fusion system of a tree of groups. Finally, we note that by Theorem~\ref{Op'}, the fusion systems corresponding to subgroups of $\Gamma_{7'}(\F)$ are also saturated.
\end{proof}

\begin{Thm}\label{Some are exotic}
Let $p \geq 5$, $S$ be a Sylow $p$-subgroup of $\Gee_2(p)$ and $\F$ be a saturated fusion system on $S$ with $O_p(\F)=1$. Then either $\F$ is isomorphic to the fusion system of $\Gee_2(p)$ or $p \le 7$ and one of the following holds:
\begin{itemize}
\item[(a)] $\F$ is exotic;
\item[(b)] $\F$ is isomorphic to the fusion system of one of the simple groups listed in Theorem \ref{t:simpfus} parts (a) and (b);
\item[(c)] $\F=\F_5^1$ and $\F$ is isomorphic to the fusion system of $\Aut(\HN)$.
\end{itemize}
\end{Thm}

\begin{proof}
Suppose that $\F=\F_S(G)$ for some finite group $G$  with $S \in \Syl_p(G)$. Since $\F_S(G) = \F_{\overline S}(\overline G)$ where $\overline G= G/O_{p'}(G)$, we may suppose that  $O_{p'}(G)=1$. Let $N$ be a minimal normal subgroup of $G$. Then $1 \neq S \cap N$ is a normal subgroup of $S$ and so  $Z(S) \le N$.

If $W \in \W\cap \E$,  then $p=7$ and $W=\langle Z(S)^{\Aut_\F(W)}\rangle =\langle Z(S)^{M}\rangle\le N$ where $M= N_G(W)$. Hence $$WZ_4=W[W,Q]\le W[W,S] =\langle W^S\rangle\le N.$$ Thus $(N \cap M)C_G(W)/C_G(W)$ contains a Sylow $p$-subgroup  of $M/C_G(W)$ and so, as $N \cap M$ is normal in $M$ and $M/C_G(W) \cong \SL_2(7)$, $(N\cap M)C_G(W)=M$.  So $\F_S(NS)$ must  contain one of the fusion systems  $O^{7'}(\F_7^1(j))$ for $j \subseteq I $ as in Notation~\ref{not}. But by Lemma \ref{l:wine}, $\Out_{\F_S(NS)}(S)$ satisfies $[S,\Out_{NS}(S)]= S$ and this means that $S \le N$.

 If  $\E \cap \W=\emptyset$, then $\E=\{Q,R\}$ by Lemma~\ref{l53}. Thus  $Z_2= Z(R) \le \langle Z(S)^{\Aut_\F(R)} \rangle \le N$, and then $Q=\langle Z(R)^{\Aut_\F(Q)}\rangle \le N$. Now $Z_4\le N$ and so $R= \langle Z_4^{\Aut_\F(R)}\rangle \le N$.  Thus $S =QR\le N$.

 We have shown that for all the fusion systems under investigation, we have $S\in \Syl_p(N)$.  Plainly $N$ is non-abelian and so $N$ is a direct product of isomorphic non-abelian simple groups. Therefore, as $Z(S)$ has order $p$, we have that $N$ is simple and that $G$ is almost simple. Since $S \in \Syl_p(N)$, Theorem~\ref{t:simpfus} shows that either $N \cong \Gee_2(p)$ or $p \le 7$ and $N$ is one of the sporadic simple groups $\Ly$, $\HN$, $\BM$ or $\Mo$.   Furthermore, in all  cases except for $N \cong \HN$ we have $\Out(N)=1$ and so either $G=N$ or $G = \Aut(\HN)$. It is now straight forward to match fusion systems to groups and this proves the theorem.
\end{proof}

We now complete the proof of Theorem \ref{t:main1}:

\begin{proof}[Proof of Theorem \ref{t:main1} for $p\ge 5$]
This follows on combining Theorems \ref{t:main}, \ref{They are saturated} and \ref{Some are exotic}.
\end{proof}

\section{Fusion systems on a Sylow $3$-subgroup of $\Gee_2(3)$}
We classify all saturated fusion systems on $S$ where, in this section, $S$ is the group $U$ constructed in the appendix  in the case $\mathbb{F}=\mathbb{F}_3$. For $\alpha$ in the root system of $\Gee_2$, we use $x_\alpha$ to denote $x_\alpha (1)$.  Set $$Q_1=\langle x_\beta, x_{\alpha+\beta}, x_{\alpha+2\beta}, x_{\alpha+3\beta}, x_{2\alpha+3\beta} \rangle \mbox{ and } Q_2=\langle x_\alpha, x_{\alpha+\beta}, x_{\alpha+2\beta}, x_{\alpha+3\beta}, x_{2\alpha+3\beta} \rangle.  $$ In particular we note that $S$ has order $3^6$ and $Q_1$ and $Q_2$ have order $3^5.$

\begin{Lem}\label{l:magma}
Suppose that  $G=\Gee_2(3)$, $G_1=\Aut(G)$, $S \in \Syl_3(G)$, $B=N_G(S)$ and $B_1=N_{G_1}(S)$. Then \begin{itemize}
\item[(a)] $Q_1$ and $Q_2$ are isomorphic to $3^2 \times 3^{1+2}$ and have exponent $3$;
\item[(b)] $Q_1 \cup Q_2$ is the set of elements in $S$ of order dividing $3$;
\item[(c)] every element of $S \backslash Q_1 \cup Q_2$ has order $9$;
\item[(d)] if $M$ is a maximal subgroup of $S$ then either $M \in \{Q_1,Q_2\}$ or $M' \ge Z(S)$;
\item[(e)] $[Q_i,S,S] \not \le \Phi(Q_i)$ for $i=1,2$;
\item[(f)] $[Z(Q_i),S] \not \le \Phi(Q_i)$ for $i=1,2$;
\item[(g)] $|\Aut(S)|=2^3\cdot 3^{10}$ and a Sylow $2$-subgroup of $\Aut(S)$ is conjugate to a subgroup of $\Aut_{B_1}(S)$;
\item[(h)] if $t \in \Aut_B(S)$ has order $2$ then $C_{Q_i}(t)$ has order $3$ or $9$ for $i=1,2$.
\end{itemize}
\end{Lem}

\begin{proof}
Some of these results can be found in \cite[Lemma 6.5]{PaR}, and others are well-known. They are also elementary to produce using {\sc Magma} \cite{magma}.
\end{proof}

In this section we complete the proof of Theorem \ref{t:main1} by proving the following result.

\begin{Thm}\label{t:G23Thm}
Let $S$ be a Sylow $3$-subgroup of $\Gee_2(3)$ and $\F$ be a saturated fusion system on $S$ with $O_3(\F)=1$. Then $\F$ is isomorphic  to the fusion system of $\Gee_2(3)$ or $\Aut(\Gee_2(3))$.
\end{Thm}

Assume that  $\F$ is a saturated fusion system on $S$. To prove Theorem~\ref{t:G23Thm} it suffices to demonstrate that up to isomorphism there are exactly two possible fusion systems on $S$ with $O_3(\F)=1$.

\begin{Lem}\label{l:g231}
 Suppose that $E \le S$ is an $\F$-essential subgroup of $\mathcal F$.  Then $E \le Q_1$ or $E\le Q_2$.
\end{Lem}

\begin{proof}
Suppose that the claim is false.  We first   examine  the possibility that $E \cap Q_1 = E \cap Q_2$.
In this case $|E/(E \cap Q_1\cap Q_2)|=3$ and every element of $E \setminus Q_1$ has order $9$. Thus $$E \cap Q_1 = E\cap Q_2=E\cap Q_1 \cap Q_2= \Omega_1(E)$$ and this group has index $3$ in $E$.
Because $E$ is centric, $E \cap Q_1 \ge Z(S)$. Since $E \ge Z(S)$ and $[S,Q_1\cap Q_2]=Z(S)$, $E$ is normalized by $Q_1\cap Q_2$ and so $$[E,Q_1\cap Q_2]\le E \cap Q_1 \cap Q_2= \Omega_1(E).$$
Furthermore, as $\Omega_1(E) \le Q_1 \cap Q_2$ and $Q_1 \cap Q_2$ is abelian, $[\Omega_1(E),Q_1\cap Q_2] =1$. Hence $$1 \unlhd \Omega_1(E) \unlhd E$$ is an $\Aut_{Q_1 \cap Q_2}(E)$-invariant chain and we conclude from Lemma \ref{l:outinj} that $\Aut_{Q_1 \cap Q_2}(E) \le O_p(\Aut_\F(E)) = \Inn(E)$.  Thus $Q_1 \cap Q_2 \le E$ and $$E \text{  is normalized by }S=Q_1Q_2.$$ Since $Z(Q_1) \le Q_1 \cap Q_2 \le E$, we now have $Z(E) \le C_S(Z(Q_i)) = Q_i$ and so $Z(S) \le Z(E) \le Q_1 \cap Q_2$. Suppose that $Z(E) > Z(S)$. Since, for $i=1,2$, $E$ does not centralizes $Z(Q_i)$,  and $|Z(Q_i)|= 3^3$, we have $Q_1 \cap Q_2= Z(Q_1)Z(E)= Z(Q_2)Z(E) $. But then \begin{eqnarray*}[Q_1\cap Q_2,E]&=& [Z(Q_1)Z(E),E]=[Z(Q_1),E]=[Z(Q_1),Q_1E]\\&=& [Z(Q_1),Q_1Q_2]=[Z(Q_1),Q_2]= \Phi(Q_2)\end{eqnarray*}
and
\begin{eqnarray*}[Q_1\cap Q_2,E]&=& [Z(Q_2)Z(E),E]=[Z(Q_2),E]=[Z(Q_2),Q_2E]\\&=& [Z(Q_2),Q_1Q_2]=[Z(Q_2),Q_1]= \Phi(Q_1)\end{eqnarray*}
 whereas we know $\Phi(Q_1)\ne \Phi(Q_2)$.
This contradiction shows that  $Z(E)= Z(S)$. Now, recalling that $\Omega_1(E)= Q_1 \cap E = E \cap Q_1\cap Q_2$,  we have \begin{eqnarray*}[E,Q_1]&\le& \Omega_1(E)\cr        [\Omega_1(E),Q_1]&\le&[Q_1\cap Q_2,Q_1]= \Phi(Q_1) \le Z(E)\text{ and }\cr   [Z(E),Q_1]&=&1.\end{eqnarray*}  Hence

$$1 \unlhd Z(E) \unlhd \Omega_1(E) \unlhd E$$ is an $\Aut_{Q_1}(E)$-invariant chain so $\Aut_{Q_1}(E) \le O_p(\Aut_\F(E)) = \Inn(E)$ which means that $Q_1 \le E$ and $S=E$, a contradiction. Hence  $E \cap Q_1 \not = E\cap Q_2$.

Suppose $E \cap Q_1 \not \le Q_2$. As $Q_1' \le Z(S) \le E$, $Q_1$ normalizes $E\cap Q_1$ and therefore $EQ_1=S$ normalizes $E\cap Q_1$. It follows that
  $$(E\cap Q_1)Z(Q_1)/Z(Q_1)\ge  C_{Q_1/Z(Q_1)}(S) =(Q_1 \cap Q_2)/Z(Q_1)$$ and this  implies that $$(E\cap Q_1)Z(Q_1)= Q_1$$ and $EZ(Q_1)= S$. Now noting that $ E \ge Z(S)$ and $|Z(Q_1):Z(S)|=3$ yields $|E|= 3^5$.  Thus $E$ is a maximal subgroup of $S$. By Lemma \ref{l:magma}(d), $E' \ge Z(S)$.
Thus
\begin{eqnarray*}[E,S]&=& [E,EZ(Q_1)]= E'[E,Z(Q_1)] \cr&\le & E' [S,Z(Q_1)]\le E'Z(S)=E'.\end{eqnarray*}
 Hence $\Aut_S(E)$ centralizes $E/E'$ and this means that $\Aut_S(E) \le O_p(\Aut_\F(E))$, a contradiction.

\end{proof}

\begin{Lem}\label{l:q1q2not}
We have $Q_1 \cap Q_2$ is not $\F$-essential.
\end{Lem}

\begin{proof}
Suppose to the contrary that $E= Q_1 \cap Q_2$ is $\F$-essential.  Then $\Out_\F(E)$ has a strongly $3$-embedded subgroup. Since $E$ is normalized by $S$,  $\Out_S(E)$ is elementary abelian of order $9$. It follows from \cite[Theorem 7.6.1]{GLS3} that, setting $X= \Out_\F(E)$, $O^{3'}(X/Z(X))$ is isomorphic to one of $\PSL_2(9)$, $\PSL_3(4)$, or $\Mat(11)$. Because the latter two groups have order which does not divide $|\GL_4(3)|$, we conclude that $O^{3'}(X/Z(X)) \cong \PSL_2(9)$. Since $C_{E}(S)= Z(S)$ has order $9$, we deduce that $O^{3'}(X)\cong \SL_2(9)$. But then, by \cite[Theorem 4.9]{AKO}, $N_{\F}(E)$ is realized by a group which contains $3^4:\SL_2(9)$ and this group has Sylow $3$-subgroups of exponent $3$, a contradiction.
\end{proof}

\begin{Lem}\label{l:AutSG23} The following hold.
\begin{enumerate}
\item [(a)] $\Out_\mathcal F(S)$ is a subgroup of $\Dih(8)$;
 \item [(b)]  $N_{\Aut_\mathcal F(S)}(Q_1)=N_{\Aut_\mathcal F(S)}(Q_2)$ has index at most $2$ in $\Aut_\F(S)$ and   $N_{\Aut_\mathcal F(S)}(Q_1)/\Inn(S)$ is elementary abelian of order at most $4$; and
 \item [(c)]if $t \in N_{\Aut_\mathcal F(S)}(Q_1)$ is an involution, then $C_{Q_i}(t)$ has order either $3$ or $3^2$.
     \end{enumerate}
\end{Lem}

\begin{proof}
Parts (a) and (b) follow from Lemma \ref{l:magma}(g) while part (c) follows from Lemma \ref{l:magma}(h).
\end{proof}

\begin{Lem}\label{l:MCLapp}    Suppose that for $i=1$ or $2$, $E < Q_i$ is $\F$-essential. Then $|E|= 3^4$, $E$ is elementary abelian,   $O^{3'}(\Aut_\mathcal F(E)) \cong \SL_2(3)$ and  $|[E,t]|=9$ for $t \in Z(O^{3'}(\Aut_\mathcal F(E)))$ of order $2$. \end{Lem}

\begin{proof}
Without loss of generality, assume that $E \le Q_1$.  Then, by Lemma~\ref{l:q1q2not}, $E \ne Q_1 \cap Q_2$.  Since $E \le Q_1$, $Z(Q_1) \le E$ and, as $E$ is centric, $E> Z(Q_1)$.  Since, by assumption, $E \ne Q_1$ and using $Q_1$ has exponent $3$, we obtain $E$ is elementary abelian of order $3^4$. As $E \ne Q_1 \cap Q_2$, we have $N_{S}(E)= Q_1$ and so $\Aut_S(E)= \Aut_{Q_1}(E)$ has order $3$.  Since $[E,Q_1]= \Phi(Q_1)$ has order $3$ we may apply the main result of \cite{McL} to see that $O^{3'}(\Aut_\mathcal F(E)) \cong \SL_2(3)$ and that $|[E,t]|=9$ for $t \in Z(O^{3'}(\Aut_\mathcal F(E)) )$ of order $2$.
\end{proof}

\begin{Lem}\label{l:qiess} Suppose that for $i=1$ or $2$, $E < Q_i$ is $\F$-essential.  Then $Q_i$ is $\F$-essential.
\end{Lem}

\begin{proof} Assume that $Q_i$ is not  $\F$-essential.  By Lemma \ref{l:MCLapp} there exists $t \in Z(O^{3'}(\Aut_\mathcal F(E)))$ of order $2$.  Then $t$ normalizes $\Aut_{S}(E)=  Q_i/Z(E)$ and hence lifts to $\tau \in \Aut_\mathcal F(Q_i)$ by saturation. Since $Q_i$ is not $\F$-essential, Lemma \ref{l:g231} implies that $\tau = \sigma|_{Q_i}$  for some $\sigma \in \Aut_\F(S)$.  Now $|C_{Q_i}(\tau)|= |Q_i/E||C_E(t)|=27$ by Lemma~\ref{l:MCLapp}. On the other hand, by Lemma~\ref{l:AutSG23} we have $|C_{Q_i}(\sigma|_{Q_i})| \le 9$, a contradiction. Hence $Q_i$ is $\F$-essential.
\end{proof}

\begin{Lem}\label{l:QiessStruct}
Suppose that $Q_i$ is $\F$-essential for some $1\le i \le 2$. Then $\Out_\F(Q_i)$ acts faithfully on $Q_i/Z(Q_i)$. In particular,  $O^{3'}(\Out_\F(Q_i))\cong \SL_2(3)$  and $\Out_\F(Q)$ embeds into $\GL_2(3)$.
\end{Lem}

\begin{proof} We may as well suppose that $i=1$. Let $X= C_{\Aut_\F(Q_1)}(Q_1/Z(Q_1))$, we will show that $X= \Inn(Q_1)$. Since $\Out_S(Q_1)$ acts faithfully on $Q_1/Z(Q_1)$ and $\Out_S(Q_1)$ has order $3$, we have $X/\Inn(Q_1)$ has $3'$-order. Thus $$Q_1/\Phi(Q_1) = [Q_1/\Phi(Q_1),X] \times C_{Q_1/\Phi(Q_1)}(X)$$ by coprime action.  Since $X$ is normal in $\Aut_\F(Q_1)$, this is a non-trivial  decomposition which is $\Aut_S(Q_1)$-invariant.   Suppose that $C_{Q_1/\Phi(Q_1)}(X) \cap Z(Q_1)/\Phi(Q_1) \ne 1$.  Then $$Z(Q_1)/\Phi(Q_1) =  (C_{Q_1/\Phi(Q_1)}(X) \cap Z(Q_1)/\Phi(Q_1))\times [Q_1/\Phi(Q_1),X]$$ as $[Q_1/\Phi(Q_1),X]=[Z(Q_1)/\Phi(Q_1),X]$. This decomposition is also non-trivial and we deduce that $\Aut_S(Q_1)$ centralizes $Z(Q_1)/\Phi(Q_1)$ which contradicts Lemma \ref{l:magma} (f).  Hence $C_{Q_1/\Phi(Q_1)}(X) \cap Z(Q_1)/\Phi(Q_1) =1$ and so $[Q_1/\Phi(Q_1),X]=Z(Q_1)/\Phi(Q_1)$.
Now  $$|C_{Q_1/\Phi(Q_1)}(X)|=|[Q_1/\Phi(Q_1),X]|=9.$$ Thus $[Q_1,S,S]\le \Phi(Q_1)$ against Lemma \ref{l:magma}(e). Hence $X= \Inn(Q_1)$.  Therefore, $\Out_\F(Q_1)$ is isomorphic to a subgroup of $\GL_2(3)$ and, as $\Out_S(Q_1)$ is not normal in $\Out_\F(Q_1)$, this proves the lemma.
\end{proof}

\begin{Prop} \label{p:justQ1orQ2} If $E \le S$ is  $\F$-essential, then $E\in\{Q_1,Q_2\}$.
\end{Prop}

\begin{proof} Suppose that $E$ is  $\F$-essential and that $E \not \in \{Q_1,Q_2\}$.  Then by Lemma~\ref{l:g231}, without loss of generality we may assume  $Z(Q_1)<E < Q_1$. Then by Lemma~\ref{l:qiess}, $Q_1 $ is $\F$-essential and, by Lemma~\ref{l:QiessStruct}, $\Aut_\F(Q_1)$   acts transitively on the maximal subgroups  of $Q_1$ containing $Z(Q_1)$. In particular, $E$ is $\F$-conjugate to $Q_1 \cap Q_2$ and this contradicts $E$ being fully $\F$-normalized.  This contradiction shows that if $E$ is $\F$-essential, then $E \in\{Q_1,Q_2\}$.
  \end{proof}

\begin{Lem}\label{l:structg23} We have $\Aut_\F(Q_1) \cong \GL_2(3)\cong \Aut_\F(Q_2)$ and $\Out_\F(S)$ is either elementary abelian of order $4$ or dihedral of order $8$.
\end{Lem}

\begin{proof} By Proposition~\ref{p:justQ1orQ2}, we may assume that $Q_1$ is $\F$-essential. If $Q_1^\F$ is the only class of essential subgroups, then, as $O_3(\F)= 1$, we must have that $\Aut_\F(S)$ has an element $\alpha$ which does not normalize $Q_1$. But then $Q_1\alpha = Q_2$.  It follows that either $Q_1$ and $Q_2$ are not $\F$-conjugate and are both $\F$-essential or that they are $\F$-conjugate.  Thus, by Lemma~\ref{l:QiessStruct}, $$O^{3'}(\Out_\F(Q_1)) \cong    O^{3'}(\Out_\F(Q_2))
\cong \SL_2(3).$$ For  $1\le i\le 2$, let $\tau_i \in \Aut_\F(Q_i)$ project to an involution in $Z(O^{3'}(\Out_\F(Q_i)))$. Then $\tau_i$ lifts to $\hat \tau_i \in \Aut_\F(S_i)$ of order $2$. Furthermore, on $S/(Q_1\cap Q_2)$, these maps normalize both $Q_1/(Q_1\cap Q_2)$  and $Q_2/(Q_1\cap Q_2)$ with  $\hat \tau_1$ centralizing $S/Q_1$ and inverting $Q_1/(Q_1 \cap Q_2)$ whereas $\hat \tau_2$ centralizes $S/Q_2$ and inverts $Q_2/(Q_1\cap Q_2)$ by Lemma \ref{l:QiessStruct}. It follows that $\Out_\F(S) \ge \langle \hat \tau_1,\hat \tau_2\rangle \Inn(S)/\Inn(S)$ which is elementary abelian of order $4$. Thus \begin{eqnarray*}\Out_\F(Q_1)&=&\langle \hat \tau_2|_{Q_1}, O^{3'}(\Aut_\F(Q_1))\rangle/\Inn(Q_1) \cong \GL_2(3), \text{ and }\\
\Out_\F(Q_2)&=&
 \langle \hat \tau_1|_{Q_2}, O^{3'}(\Aut_\F(Q_2))\rangle / \Inn(Q_2)\cong \GL_2(3).\end{eqnarray*} Finally, either $\Out_\F(S) = \langle \hat \tau_1,\hat \tau_2\rangle \Inn(S)/\Inn(S)$ or $\Out_\F(S) \cong \Dih(8)$ by Lemma~\ref{l:AutSG23}.

\end{proof}

\begin{proof}[Proof of Theorem~\ref{t:G23Thm}]  By Lemmas~\ref{l:AutSG23} and \ref{l:structg23}   $\Aut_\F(S)$ has a subgroup $ \Aut_\F^0(S)$ of index at most $2$ which has order $2^23^4$ with elementary abelian Sylow $2$-subgroups. This subgroup normalizes both $Q_1$ and $Q_2$ and is uniquely determined up to conjugacy in $\Aut(S)$. We fix it once and for all. Let $N = N_{\Aut(S)}(\Aut_\F^0(S))$. Then $N$ has a subgroup $N^0$ of index $2$ which normalizes both $Q_1$ and $Q_2$.  Now we calculate using {\sc Magma} that, for $i=1,2$,  the restriction $K$ of   $\Aut_\F^0(S) $ to $Q_i$ is contained in exactly three subgroups $X$ of $\Aut(Q_i)$ containing $\Inn(Q_i)$ which have $X/\Inn(Q_i) \cong \GL_2(3)$. Since $ K$ must coincide with $N_{\Aut_\F(Q_i)}(\Aut_S(Q_i)) \sim 3^{1+2}_+.2^2$ we see, using Lemma~\ref{l:structg23}, that there are exactly three candidates for the subgroup $\Aut_\F(Q_1)$ of $\Aut(Q_1)$ and also three candidates for the subgroup $\Aut_\F(Q_2)$ of $\Aut(Q_2)$. Next we calculate that $N_0$ restricted to $Q_1$ conjugates these three candidates for $\Aut_\F(Q_1)$ together and thus we have a subgroup $N^1$ of index $3$ in $N_0$ which normalizes $\Aut_\F(Q_1)$, $\Aut_\F^0(S)$. We calculate that the restriction of $N^1$ to $Q_2$ acts transitively on the three candidates for $\Aut_\F(Q_2)$. Thus the triple $\Aut_\F^0(S)$, $\Aut_\F(Q_1)$, $\Aut_\F(Q_2)$ is uniquely determine up to $\Aut(S)$ conjugacy. If $\Out_\F(S)$ has order $4$, then  $\Aut_\F^0(S)= \Aut_\F(S)$ and $\F$ is uniquely determined up to isomorphism.  If $\Aut_\F(S)> \Aut_\F(S)$, we check that $\Out_\F^0(S)$ is contained in a unique subgroup of order $8$ which conjugates $\Aut_\F(Q_1)$ to $\Aut_\F(Q_2)$. This proves that  there are exactly two saturated fusion systems on $S$ up to isomorphism.  Since $\F_S({\Gee_2(3)})$ and $\F_S({\Aut(\Gee_2(3))})$ provide examples of fusion systems, we have completed the proof of the theorem.
\end{proof}

\section{Appendix}
\subsection{Construction of a Sylow $p$-subgroup of $\Gee_2(q)$}
Following Wilson \cite{Wilson}, we construct the Sylow $p$-subgroup $U$ of  $\Gee_2(q)$. Let $\mathbb{F}$ denote a finite field of order $q=p^f$ for some prime $p \ge 3$ and let  $\mathbb{O}=\mathbb{F}[i_0,i_1,i_2,i_3,i_4,i_5,i_6]$ be the octonion algebra where, taking subscripts modulo 7, $i_t,i_{t+1},i_{t+3}$ satisfy same multiplication rules as $i,j,k$ in the quaternions. For completeness, we give the multiplication table below:

\begin{center}
\label{my-label}
\begin{tabular}{l|rrrrrrr|}
  & $i_0$  & $i_1$  & $i_2$  & $i_3$  & $i_4$  & $i_5$  & $i_6$  \\ \hline
$i_0$ & $-1$   & $i_3$  & $i_6$  & $-i_1$ & $i_5$  & $-i_4$ & $-i_2$ \\
$i_1$ & $-i_3$ & $-1$   & $i_4$  & $i_0$  & $-i_2$ & $i_6$  & $-i_5$ \\
$i_2$ & $-i_6$ & $-i_4$ & $-1$   & $i_5$  & $i_1$  & $-i_3$ & $i_0$  \\
$i_3$ & $i_1$  & $-i_0$ & $-i_5$ & $-1$   & $i_6$  & $i_2$  & $-i_4$ \\
$i_4$ & $-i_5$ & $i_2$  & $-i_1$ & $-i_6$ & $-1$   & $i_0$  & $i_3$  \\
$i_5$ & $i_4$  & $-i_6$ & $i_3$  & $-i_2$ & $-i_0$ & $-1$   & $i_1$  \\
$i_6$ & $i_2$  & $i_5$  & $-i_0$ & $i_4$  & $-i_3$ & $-i_1$ & $-1$   \\ \hline
\end{tabular}

\end{center}

Wilson  defines a additional  basis for $\mathbb{O}$ as follows:  first choose $a,b \in \mathbb{F}$ such that $a^2+b^2=-1$, $b \ne 0$ and define a new basis $\{y_1,y_2,\ldots,y_8\}$ of $\mathbb{O}$ by setting:

$$\begin{array}{cc}
2y_1=i_4+ai_6+bi_0, & 2y_8=i_4-ai_6-bi_0 \\
2y_2=i_2+bi_3+ai_5, & 2y_7=i_2-bi_3-ai_5, \\
2y_3=i_1-bi_6+ai_0, & 2y_6=i_1+bi_6-ai_0, \\
2y_4=1+ai_3-bi_5, & 2y_5=1-ai_3+bi_5.
\end{array}$$

The new multiplication table is given on \cite[p. 123]{Wilson} and involves coefficients  $\pm 1$.  Moreover, on  \cite[p. 124]{Wilson} Wilson gives a maximal unipotent subgroup $U$ of $\Gee_2(q)$   in terms of its action with respect to this basis. Let   $\mathcal{R}=\{\alpha, \beta, \alpha+\beta, \alpha+2\beta, \alpha+3\beta, 2\alpha+3\beta\}$ be a set of positive roots for the $\Gee_2$ root system. Then,  for $\lambda \in \mathbb F$, $U$ is generated by the  matrices
$$\begin{array}{rrr}
x_{2\alpha+3\beta}(\lambda)= \left( \begin{smallmatrix}
1 & 0 & 0 & 0 & 0 & 0 & 0 & 0  \\
0 & 1 & 0 & 0 & 0 & 0 & 0 & 0  \\
0 & 0 & 1 & 0 & 0 & 0 & 0 & 0  \\
0 & 0 & 0 & 1 & 0 & 0 & 0 & 0  \\
0 & 0 & 0 & 0 & 1 & 0 & 0 & 0  \\
0 & 0 & 0 & 0 & 0 & 1 & 0 & 0  \\
-\lambda & 0 & 0 & 0 & 0 & 0 & 1 & 0  \\
0 & \lambda & 0 & 0 & 0 & 0 & 0 & 1  \\
 \end{smallmatrix} \right), & x_{\alpha+3\beta}(\lambda)= \left( \begin{smallmatrix}
1 & 0 & 0 & 0 & 0 & 0 & 0 & 0  \\
0 & 1 & 0 & 0 & 0 & 0 & 0 & 0  \\
0 & 0 & 1 & 0 & 0 & 0 & 0 & 0  \\
0 & 0 & 0 & 1 & 0 & 0 & 0 & 0  \\
0 & 0 & 0 & 0 & 1 & 0 & 0 & 0  \\
-\lambda & 0 & 0 & 0 & 0 & 1 & 0 & 0  \\
0 & 0 & 0 & 0 & 0 & 0 & 1 & 0  \\
0 & 0 & \lambda & 0 & 0 & 0 & 0 & 1  \\
 \end{smallmatrix} \right), & x_{\alpha+2\beta}(\lambda)= \left( \begin{smallmatrix}
1 & 0 & 0 & 0 & 0 & 0 & 0 & 0  \\
0 & 1 & 0 & 0 & 0 & 0 & 0 & 0  \\
0 & 0 & 1 & 0 & 0 & 0 & 0 & 0  \\
-\lambda & 0 & 0 & 1 & 0 & 0 & 0 & 0  \\
\lambda & 0 & 0 & 0 & 1 & 0 & 0 & 0  \\
0 & -\lambda & 0 & 0 & 0 & 1 & 0 & 0  \\
0 & 0 & \lambda & 0 & 0 & 0 & 1 & 0  \\
\lambda^2 & 0 & 0 & -\lambda& \lambda & 0 & 0 & 1  \\
 \end{smallmatrix} \right), \\
&& \\
x_{\alpha+\beta}(\lambda)= \left( \begin{smallmatrix}
1 & 0 & 0 & 0 & 0 & 0 & 0 & 0  \\
0 & 1 & 0 & 0 & 0 & 0 & 0 & 0  \\
-\lambda & 0 & 1 & 0 & 0 & 0 & 0 & 0  \\
0 & -\lambda & 0 & 1 & 0 & 0 & 0 & 0  \\
0 & \lambda & 0 & 0 & 1 & 0 & 0 & 0  \\
0 & 0 & 0 & 0 & 0 & 1 & 0 & 0  \\
0 & \lambda^2 & 0 & -\lambda & \lambda & 0 & 1 & 0  \\
0 & 0 & 0 & 0 & 0 & \lambda & 0 & 1  \\
 \end{smallmatrix} \right), & x_{\beta}(\lambda)= \left( \begin{smallmatrix}
1 & 0 & 0 & 0 & 0 & 0 & 0 & 0  \\
0 & 1 & 0 & 0 & 0 & 0 & 0 & 0  \\
0 & -\lambda & 1 & 0 & 0 & 0 & 0 & 0  \\
0 & 0 & 0 & 1 & 0 & 0 & 0 & 0  \\
0 & 0 & 0 & 0 & 1 & 0 & 0 & 0  \\
0 & 0 & 0 & 0 & 0 & 1 & 0 & 0  \\
0 & 0 & 0 & 0 & 0 & \lambda & 1 & 0  \\
0 & 0 & 0 & 0 & 0 & 0 & 0 & 1  \\
 \end{smallmatrix} \right), & x_{\alpha}(\lambda)= \left( \begin{smallmatrix}
1 & 0 & 0 & 0 & 0 & 0 & 0 & 0  \\
-\lambda& 1 & 0 & 0 & 0 & 0 & 0 & 0  \\
0 & 0 & 1 & 0 & 0 & 0 & 0 & 0  \\
0 & 0 & \lambda & 1 & 0 & 0 & 0 & 0  \\
0 & 0 & -\lambda & 0 & 1 & 0 & 0 & 0  \\
0 & 0 & \lambda^2 & \lambda & -\lambda & 1 & 0 & 0  \\
0 & 0 & 0 & 0 & 0 & 0 & 1 & 0  \\
0 & 0 & 0 & 0 & 0 & 0 & \lambda& 1  \\
 \end{smallmatrix} \right).
\end{array}
$$
 For $\theta \in \mathcal R$, the elements $x_\theta(\lambda)$, $\lambda \in \mathbb F$ generate the root group corresponding to $\theta$.  It is straightforward to verify that the following relations hold for all $\lambda, \mu \in \mathbb F$:
$$\begin{array}{rl}
%(x_r)^p & = 1  \mbox{ for all } r \in \mathcal{R} \\
\left[x_\beta(\lambda), x_\alpha(\mu) \right] & = x_{2\alpha+3\beta}(2\mu^3\lambda^2)   x_{\alpha+3\beta}(-\mu^3\lambda)   x_{\alpha+2\beta}(\mu^2 \lambda)   x_{\alpha+\beta}(-\mu\lambda) \\

\left[x_{\alpha+\beta}(\lambda), x_\alpha(\mu)\right] & = x_{2\alpha+3\beta}(-3\mu\lambda^2)    x_{\alpha+3\beta}(3\mu^2\lambda) x_{\alpha+2\beta}(-2\mu\lambda) \\

\left[x_{\alpha+2\beta}(\lambda), x_\alpha(\mu)\right] & = x_{\alpha+3\beta}(-3\mu\lambda) \\

\left[x_{\alpha+3\beta}(\lambda),x_\beta(\mu) \right] & = x_{2\alpha+3\beta}(3\mu\lambda) \\
\left[x_{\alpha+2\beta}(\lambda),x_{\alpha+\beta}(\mu) \right] & = x_{2\alpha+3\beta}(-\mu\lambda) \\

\left[x_r(\lambda),x_s(\mu)\right] & = 1 \mbox{ for all other } \{r,s\} \subset \mathcal{R}.
\end{array}$$

Recall the definitions of the elements $x_1(\lambda), x_2(\lambda),\ldots, x_6(\lambda)$ in Section \ref{ss:constructs}.

\begin{Prop}
Let $\phi: \mathcal{R} \rightarrow \{1,2,3,4,5,6\}$ be the mapping which sends the tuple of elements $(\alpha, \beta, \alpha+\beta, \alpha+2\beta, \alpha+3\beta, 2\alpha+3\beta) \mbox{ to } (1,2,3,4,5,6).$  Then the induced map $U \rightarrow S$ given for each $\lambda \in \mathbb{F}$ by $x_r(\lambda) \mapsto x_{r\phi}(\lambda)$ is an isomorphism.
\end{Prop}

\begin{proof}
This is routine to check.
\end{proof}

\end{document}